%% file: main_version01b.tex
\documentclass[final,twocolumn,5p,times,10pt]{elsarticle}
\usepackage{lineno}
\modulolinenumbers[5]

\journal{Automatica}

\usepackage[utf8]{inputenc}
\usepackage[american]{babel}

\usepackage{comment}
\usepackage{amsthm}
\usepackage{amsmath}
\usepackage{amssymb}
\usepackage{amsfonts}
\usepackage{bm}
\usepackage{hhline}
\usepackage{csquotes}
\usepackage{color}
\usepackage{import}
\usepackage[ruled,vlined,linesnumbered]{algorithm2e}
\usepackage{dsfont}
\usepackage{multirow}
\usepackage{tensor}
\usepackage{subfig}
\usepackage{soul}
\usepackage[normalem]{ulem}
\usepackage{graphicx}
\usepackage{microtype}
\usepackage{dsfont}

\makeatletter
\usepackage[colorlinks=true,linkcolor=red]{hyperref}
\makeatother

\newcommand{\bi}{\begin{itemize}}\newcommand{\ei}{\end{itemize}}
\newcommand{\be}{\begin{equation}}\newcommand{\ee}{\end{equation}}
\newcommand{\bes}{\begin{equation*}}\newcommand{\ees}{\end{equation*}}
\newcommand{\bb}{\begin{bmatrix}}\newcommand{\eb}{\end{bmatrix}}
\newcommand{\bee}{\begin{enumerate}}\newcommand{\eee}{\end{enumerate}}
\newcommand{\bea}{\begin{eqnarray}}\newcommand{\eea}{\end{eqnarray}}
\newcommand{\beas}{\begin{eqnarray*}}\newcommand{\eeas}{\end{eqnarray*}}
\newcommand{\bc}{\begin{center}}\newcommand{\ec}{\end{center}}

\newtheorem{theo}{Theorem}

\newtheorem{coro}{Corollary}
\newtheorem{lemm}{Lemma}
\newtheorem{propo}{Proposition}

\newtheorem{defi}{Definition}
\newtheorem{assu}{Assumption}

\newtheorem{rema}{Remark}

\allowdisplaybreaks 

\newcommand{\GR}[1]{\textcolor{black}{#1}}
\newcommand{\AP}[1]{\textcolor{black}{#1}}

\SetKwInput{KwData}{Initialization}

\input{macros.tex}
\begin{document}

\begin{frontmatter}

\title{Set-based state estimation for discrete-time constrained nonlinear systems: an approach based on constrained zonotopes and DC programming\tnoteref{myfootnote}}
\tnotetext[mytitlenote]{This work was partially supported by the Brazilian agencies CNPq, under grants 465755/2014-3 (INCT project), 315695/2020-0, and 306564/2020-3; FAPESP, under grant 2014/50851-0;  CAPES, under the grants 001 and 88887.136349/2017-00; FAPEMIG, under grant APQ-03090-17 and by the Italian Ministry for Research in the framework of the 2017 Program for Research Projects of National Interest (PRIN) (Grant no. 2017YKXYXJ).}



\author[myfirstaddress]{Alesi A. de Paula\corref{mycorrespondingauthor}}
\cortext[mycorrespondingauthor]{Corresponding author}
\ead{alesi@ufmg.br}

\author[mysecondaddress]{Davide M. Raimondo}
\ead{davide.raimondo@unipv.it}

\author[myfirstaddress,mythirdaddress]{Guilherme V. Raffo}
\ead{raffo@ufmg.br}

\author[myfirstaddress,mythirdaddress]{Bruno O. S. Teixeira}
\ead{brunoot@ufmg.br}

\address[myfirstaddress]{Graduate Program in Electrical Engineering, Universidade Federal de Minas Gerais, Belo Horizonte, MG, Brazil}
\address[mysecondaddress]{Department of Electrical, Computer and Biomedical Engineering, University of Pavia, Italy}
\address[mythirdaddress]{Department of Electronic Engineering, Universidade Federal de Minas Gerais, Belo Horizonte, MG, Brazil}

\begin{abstract}
This paper proposes a new state estimator for discrete-time nonlinear dynamical systems with unknown-but-bounded uncertainties and state linear inequality and nonlinear equality constraints. Our algorithm is based on constrained zonotopes (CZs) and on a DC programming approach (DC stands for difference of convex functions). Recently, mean value extension and first-order Taylor extension have been adapted from zonotopes to propagate CZs over nonlinear mappings. Although the resulting algorithms (called CZMV and CZFO) reach better precision than the original zonotopic versions, they carry the sensitivity to the wrapping and dependency effects inherited from interval arithmetic.
These interval issues can be mitigated with DC programming since the \AP{approximation error bounds are obtained solving} optimization problems. A direct benefit of this technique is the elimination of the dependency effect. 
Our set-membership filter (called CZDC) offers an alternative solution to CZMV and CZFO.
In order to demonstrate the effectiveness of the proposed approach, CZDC is experimented over two numerical examples.
\end{abstract}


\begin{keyword}
Constrained zonotopes, Set-membership filtering, Discrete-time nonlinear systems, DC programming, Nonlinear invariants, linear inequality constraints
\end{keyword}

\end{frontmatter}


\input{01Introduction_version01b}
\input{02ProblemFormulation_version01b}
\input{03Preliminaries_version01b}
\input{04CZDC_version01b}
\input{05NumericalExamples_version01b}
\input{06Conclusions_version01b}

\bibliography{my_bib_version01b}

\end{document}

%% file: 01Introduction_version01b.tex
\section{Introduction} \label{sec_introduction}



Set-based techniques have been investigated in the literature to solve problems involving parameter estimation 
\cite{bravo_bounded2006,bravo2008algorithm},
state estimation 
\cite{rego2021set,alesi2022CSM}, fault diagnosis 
\cite{xu2015set,hast2015detection}, control design 
\cite{bravo_robust2006,mesbah2016stochastic}, among others. In these cases, sets are used to represent unknown-but-bounded uncertainties. 
The success of set-based techniques has been exemplified by means of their use in many \GR{applications such as }
fault detection and isolation for industrial applications 
\cite{hast2015detection}, fault diagnosis for wind turbines 
\cite{tabatabaeipour_fault2012}, dynamic robot localization and mapping \cite{di2004set}, active localization of static features for mobile robots using range-only sensors \cite{grocholsky2006extensive}, vehicle state estimation \cite{ifqir2020zonotopic}, and robot-assisted dressing \cite{li2021set}.

Set-based filtering can be split in two branches: interval observers and set-membership observers \cite{pourasghar2019interval}. We here focus on the second one, whose main difference is the presence of \emph{intersection} among sets to combine forecast \AP{and} measurement sets. 
\GR{Recently}, constrained zonotopes (CZs) have motivated new advances in set membership since they can in principle represent any convex polytope efficiently. The class of CZs extends zonotopes (centrally symmetric convex polytopes) by introducing linear equality constraints. A direct gain of this extension is to propagate asymmetric polytopes, keeping the computational advantages of zonotopes.
Also, from the introduction of the generalized intersection among CZs, which can be computed exactly, the loss of precision with respect to the zonotopic intersection\GR{, which demands in general an approximation,} can be in principle eliminated.


The original paper on CZs \cite{scott2016} \AP{has considered} state estimation and fault diagnosis for state-space linear uncertain systems. The authors have shown that their state estimator reaches better precision and detection ratio over other guaranteed estimators 
\cite{chisci1996recursive,alamo_t2005} 
at the cost of a slight \GR{increase} of processing time. Motivated by these benefits, the algorithm of \cite{scott2016} has been extended to cover more general cases as in \cite{rego2022joint}. 
In particular, recent contributions for nonlinear systems have been achieved with CZs in \cite{rego2021set,rego2020guaranteed}. \cite{rego2020guaranteed} \AP{has considered} linear output equations, while \cite{rego2021set} \AP{has extended} \cite{rego2020guaranteed} to nonlinear measurement models and \AP{added} a step to enforce algebraic equations on set-based estimates. The proposed algorithms have been developed on existing methods of zonotopes, namely: mean value extension \cite{alamo_t2005} and first-order Taylor extension \cite{c_combastel2005}.

All aforementioned nonlinear methods rely on interval arithmetic to compute interval enclosures related to the approximation remainder. Therefore, the algorithms proposed in \cite{rego2021set,rego2020guaranteed} still bring up the sensitivity to the so-called wrapping and dependency effects. These effects summarize all conservatism reached by set-valued operations, with the dependency one being caused by the multioccurrence of variables, \GR{while the remaining conservatism can be caused by linear mapping and generalized intersection due to the wrapping effect} \cite{alesi2022CSM}. An alternative tool to mitigate these interval issues is the DC programming approach, where nonconvex mathematical programming problems are \AP{approximated} using convex analysis tools 
\cite{tao1997convex}.

Aiming at set-membership filtering, \cite{alamo_set2008} has proposed approximate solutions for DC programming problems where lower and upper bounds are provided to enclose the global solutions. To achieve that, the authors replace the \emph{exact} minimization and maximization problems by \emph{approximate} versions whose solutions can be obtained through the evaluation at the vertices of a convex polytope.
According to this proposal, \cite{alamo_set2008} \AP{has} modified the zonotopic filter from \cite{alamo_t2005} (based on mean value extension) to incorporate DC programming, and thereby, mitigated the conservatism caused by the wrapping and dependency effects.
The methodology of \cite{alamo_set2008} differs from interval arithmetic in many senses, among them the range of nonlinear functions is determined by using convex components and real-valued operations rather than \GR{approximate} functions and set-valued operations.

Motivated by the benefits of DC programming over interval arithmetic and of CZs over zonotopes, we here propose a new set-membership filter called CZDC. Unlike \cite{alamo_set2008}, our algorithm allows for working on asymmetric polytopes, mitigating the conservatism generated by zonotopes over intersection, and \AP{allows to account for} state linear inequality and nonlinear equality constraints. 
Such state constraints are present in many real applications such as compartmental systems (nonnegativity, conservation laws) \cite{teixeira2009state}, unit-quaternion representation (holonomic constraints) \cite{rego2021set}, and water distribution networks (physical constraints, static relations) \cite{wang2017robust}.

This paper is organized as follows. Section \ref{sec_problem} formulates the state-estimation problem involving nonlinear state-space models with state constraints. Section \ref{sec_preliminaries} introduces some preliminary results. Section \ref{sec_CZDC} presents the CZDC algorithm in detail. Afterwards, CZDC is executed and compared to the algorithms of \cite{rego2021set} over two numerical examples in Section \ref{sec_numerical_examples}. Section \ref{sec_conclusion} presents the concluding remarks.

\section*{Notation}
\AP{
The set of natural numbers is denoted as $\mathbb{N}$.
The set of positive integer numbers is denoted as $\mathbb{Z}_+$.
The set of real numbers is denoted as $\mathbb{R}$.
An $(n \times 1)$-dimensional vector and an $(n \times m)$-dimensional matrix are, respectively, denoted as $b \in \mathbb{R}^n$ and $A \in \mathbb{R}^{n \times m}$. 
An $(n \times m)$-dimensional zero matrix and an $(n \times n)$-dimensional identity matrix are, respectively, denoted as $0_{n \times m}$ and ${\rm I}_n$.
The transpose of a matrix and the diagonal matrix obtained from a vector are, respectively, denoted as $(\cdot)^{\top}$ and ${\rm diag}(\cdot)$.
The $i$th row of a matrix is denoted as $(\cdot)_{i,:}$.}

%% file: 02ProblemFormulation_version01b.tex
\section{Problem Statement} \label{sec_problem}

Consider the discrete-time nonlinear dynamical system
\begin{align}
\label{eq_process_model}
x_k &= f\left(x_{k-1},u_{k-1},w_{k-1}\right),\\
\label{eq_measurement_model}
y_k &= h\left(x_k,v_k\right),
\end{align}
where $f:\mathbb{R}^n \times \mathbb{R}^p \times \mathbb{R}^q \to \mathbb{R}^n$ and $h:\mathbb{R}^n \times \mathbb{R}^r \to \mathbb{R}^m$ are the known process dynamics and measurement equations, respectively, $u_{k-1} \in \mathbb{R}^p$ is the known deterministic input vector, $y_k \in \mathbb{R}^m$ is the measured output vector, and $x_k \in \mathbb{R}^n$ is the state vector to be estimated. We assume that $x_k$ satisfies the following nonlinear equality and linear inequality constraints:
\begin{align}
\label{eq_equality_constraints}
g\left(x_k\right) &= 0_{m_c \times 1}, \\
\label{eq_inequality_constraints}
D_kx_k &\le d_k,
\end{align}
where $g: \mathbb{R}^n \to \mathbb{R}^{m_c}$, $D_k \in \mathbb{R}^{n_c \times n}$, and $d_k \in \mathbb{R}^{n_c}$.
Regarding \eqref{eq_inequality_constraints}, we make the following assumption to enable the direct use of convex polytopes.

\begin{assu}
The inequality constraints given by \eqref{eq_inequality_constraints}, \AP{if present,} define a compact feasibility set $\mathcal{X}^{\rm F}_k \subset \mathbb{R}^n$. 
\end{assu}

The process noise $w_{k-1} \in \mathbb{R}^q$, the measurement noise $v_k \in \mathbb{R}^r$, and the initial state $x_0 \in \mathbb{R}^n$ are bounded by convex polytopes $\mathcal{W}_{k-1}$, $\mathcal{V}_k$, and $\mathcal{X}_0$. Our set-membership filter aims at estimating the state vector $x_k$ through convex polytopes $\mathcal{X}_k$ over $k \in \mathbb{Z}_+$. To achieve this goal \GR{at each $k$}, \AP{given $u_{k-1}$,} we define five steps as follows:
\begin{enumerate}
\item {\em Forecast}:\\
$\mathcal{X}_{k|k-1} \supseteq \left\{f\left(x_{k-1},u_{k-1},w_{k-1}\right):x_{k-1} \in \mathcal{X}_{k-1},w_{k-1} \in \mathcal{W}_{k-1}\right\}$;
\item {\em Data assimilation}:\\
$\breve{\mathcal{X}}_k \supseteq \left\{x_k \in \mathcal{X}_{k|k-1}: h\left(x_k,v_k\right) = y_k,v_k \in \mathcal{V}_k\right\}$;
\item {\em Admissibility}: $\check{\mathcal{X}}_k = \breve{\mathcal{X}}_k \cap \mathcal{X}^{\rm F}_k$;
\item {\em Consistency}: $\tilde{\mathcal{X}}_k \supseteq \left\{x_k \in \check{\mathcal{X}}_k: g\left(x_k\right) = 0_{m_c \times 1}\right\}$;
\item {\em Reduction}: $\mathcal{X}_k \supset \tilde{\mathcal{X}}_k$, with $\mathcal{X}_k$ being a set with lower complexity than $\tilde{\mathcal{X}}_k$.
\end{enumerate}

Steps 1, 2, and 4 will be supported by a DC programming approach to obtain tight solutions. In steps 3 and 4, the state nonlinear equality and linear inequality constraints given by \eqref{eq_equality_constraints}-\eqref{eq_inequality_constraints} are enforced on the estimator.
Step 5 corresponds to a complexity reduction for convex polytopes, which is necessary to control the demand of computational resources; see Section \ref{sec_preliminaries} for further details.

%% file: 03Preliminaries_version01b.tex
\section{Preliminaries} \label{sec_preliminaries}


\subsection{Constrained Zonotopes}

A {\em constrained zonotope} $\mathcal{X} \subset \mathbb{R}^n$ is a convex polytope represented by the generator matrix $G^{\rm x} \in \mathbb{R}^{n \times n_g}$, the center $c^{\rm x} \in \mathbb{R}^n$, and the linear equality constraints given by matrix $A^{\rm x} \in \mathbb{R}^{n_h \times n_g}$ and vector $b^{\rm x} \in \mathbb{R}^{n_h}$. 
The terms $n_g$ and $n_h$ refer to the number of generators and constraints, respectively.
Let $\mathcal{B}^{n_g} \triangleq [-1,1]^{n_g}$ be the unitary box of dimension $n_g$ and let $\mathcal{B}\left(A^{\rm x},b^{\rm x}\right) \triangleq \left\{\xi \in \mathcal{B}^{n_g}:A^{\rm x}\xi = b^{\rm x}\right\}$ be the constrained unitary box. Then, a CZ
is defined as \cite{scott2016}
\be
\label{eq_defi_cz}
\mathcal{X} \triangleq \left\{G^{\rm x},c^{\rm x},A^{\rm x},b^{\rm x}\right\} = \left\{G^{\rm x}\xi + c^{\rm x}: \xi \in \mathcal{B}\left(A^{\rm x},b^{\rm x}\right)\right\}.
\ee
For zonotopes, there are no equality constraints given by $A^{\rm x}$ and $b^{\rm x}$. In this case, we abbreviate the notation to $\mathcal{X} = \left\{G^{\rm x},c^{\rm x}\right\}$.

Let $m \in \mathbb{R}^r$, $L \in \mathbb{R}^{r \times n}$, $\mathcal{Y} = \left\{G^{\rm y},c^{\rm y},A^{\rm y},b^{\rm y}\right\} \subset \mathbb{R}^n$, $\mathcal{W} = \left\{G^{\rm w},c^{\rm w},A^{\rm w},b^{\rm w}\right\} \subset \mathbb{R}^p$, and $M \in \mathbb{R}^{p \times n}$.
The affine transformation, Minkowski sum, generalized intersection, and Cartesian product of CZs are {\em explicitly} computed as, respectively,
\begin{align}
\label{eq_aff_transf_cz}
L\mathcal{X} \oplus m &= \left\{ LG^{\rm x} , \left(Lc^{\rm x} + m\right), A^{\rm x}, b^{\rm x} \right\},
\\
\label{eq_sum_cz}
\mathcal{X} \oplus \mathcal{Y} &= 
\left\{\bb G^{\rm x} & G^{\rm y} \eb ,  \left(c^{\rm x} + c^{\rm y}\right), 
\bb A^{\rm x} & 0_{n_h \times n^{\rm y}_g} \\ 0_{n^{\rm y}_h \times n_g} & A^{\rm y} \eb,
\bb b^{\rm x} \\ b^{\rm y} \eb \right\},
\\
\label{eq_intersection_cz}
\mathcal{X} \cap_M \mathcal{W} &= 
\left\{\bb G^{\rm x} & 0_{n \times n^{\rm w}_g} \eb ,
c^{\rm x},
\bb A^{\rm x} & 0_{n_h \times n^{\rm w}_g} \\ 0_{n^{\rm w}_h \times n_g} & A^{\rm w} \\ MG^{\rm x} & -G^{\rm w} \eb,
\bb b^{\rm x} \\ b^{\rm w} \\ c^{\rm w} - Mc^{\rm x} \eb \right\},\\
\label{eq_cart_prod_cz}
\mathcal{X} \times \mathcal{W} &= 
\left\{\bb G^{\rm x} & 0_{n \times n^{\rm w}_g} \\ 0_{n^{\rm w} \times n_g} & G^{\rm w} \eb, 
\bb c^{\rm x} \\ c^{\rm w} \eb, 
\bb A^{\rm x} & 0_{n_h \times n^{\rm w}_g} \\ 0_{n^{\rm w}_h \times n_g} & A^{\rm w} \eb, \bb b^{\rm x} \\ b^{\rm w} \eb \right\}.
\end{align}

For the set operations \eqref{eq_sum_cz}-\eqref{eq_cart_prod_cz}, the number of constraints $n_h$ and generators $n_g$ for CZs increases. Recursively, this dimension growth demands an algorithm to reduce $n_h$ and $n_g$ to desired values $\varphi_c$ and $\varphi_g$, thus keeping complexity limited at the price of conservativeness (outer approximation). 
Here, we employ the algorithm proposed in \cite{scott2016}, which can be summarized in four steps: (i) rescaling; (ii) preconditioning; (iii) elimination of constraints and (partial) generators; and (iv) final elimination of generators. The first three steps lead to reduction from $n_h$ to $\varphi_c$ and from $n_g$ to $(n_g - \varphi_c)$, while the latter finishes the reduction from $(n_g - \varphi_c)$ to $\varphi_g$.

The following result is used to obtain the so-called {\em interval hull} of a CZ $\mathcal{X}$, $\Box \mathcal{X} = \left[\zeta^{\rm L},\zeta^{\rm U}\right]$, such that $\mathcal{X} \subseteq \Box \mathcal{X}$.

\begin{propo}[\cite{rego2020guaranteed}]\label{propo_int_hull_cz}
Let $\mathcal{X} = \left\{G^{\rm x},c^{\rm x},A^{\rm x},b^{\rm x}\right\} \subset \mathbb{R}^n$.
The interval hull $\left[\zeta^{\rm L},\zeta^{\rm U}\right] \supseteq \mathcal{X}$ is obtained by solving linear programs for each $i = 1,\ldots,n$:
\begin{align*}
\zeta^{\rm L}_i &\triangleq \min_{\xi} \left\{G^{\rm x}_{i,:}\xi + c^{\rm x}_i : \xi \in \mathcal{B}\left(A^{\rm x},b^{\rm x}\right)\right\},\ i = 1,\ldots,n,\\
\zeta^{\rm U}_i &\triangleq \max_{\xi} \left\{G^{\rm x}_{i,:}\xi + c^{\rm x}_i : \xi \in \mathcal{B}\left(A^{\rm x},b^{\rm x}\right)\right\},\ i = 1,\ldots,n.
\end{align*}
\end{propo}

As any box expressed in interval arithmetic, the interval hull of a CZ can be equivalently expressed in affine arithmetic doing $\Box \mathcal{X} = \left\{{\rm diag}\left({\rm rad}\left(\Box \mathcal{X}\right)\right),{\rm mid}\left(\Box \mathcal{X}\right)\right\}$, where ${\rm rad}\left(\Box \mathcal{X}\right) \triangleq \frac{1}{2}\left(\zeta^{\rm U} - \zeta^{\rm L}\right)$ and ${\rm mid}\left(\Box \mathcal{X}\right) \triangleq \frac{1}{2}\left(\zeta^{\rm L} + \zeta^{\rm U}\right)$. For interval matrices $[M] = \left\{M \in \mathbb{R}^{n \times m}: M^{\rm L} \le M \le M^{\rm U}\right\}$, we have ${\rm rad}([M]) = \frac{1}{2}\left(M^{\rm U} - M^{\rm L}\right)$ and ${\rm mid}([M]) = \frac{1}{2}\left(M^{\rm L} + M^{\rm U}\right)$, with $M^{\rm L}$ and $M^{\rm U}$ being known matrices with different values $M^{\rm L}_{i,j}$ and $M^{\rm U}_{i,j}$, respectively, for $i = 1,\ldots,n$ and $j = 1,\ldots,m$.

\subsection{DC Programming}

As shown in \cite{rego2021set}, \eqref{eq_aff_transf_cz}-\eqref{eq_intersection_cz} can be directly employed in state estimation when $f$, $h$, and $g$ given by \eqref{eq_process_model}, \eqref{eq_measurement_model}, and \eqref{eq_equality_constraints} are linear. Conversely, the nonlinear case requires some approximation of such functions to enable the use of \eqref{eq_aff_transf_cz}-\eqref{eq_intersection_cz}. Contributions to this topic have been proposed in \cite[Lemmas 1 and 2]{rego2021set} using CZs.

In this paper, a DC programming approach is used to compute linearization enclosures. This approach is convenient to reduce conservatism in comparison with interval methods based on Lagrange remainder as those proposed in \cite{c_combastel2005,althoff2008reachability}, which concentrate the linearization error in the quadratic term of a truncated Taylor series. For an in-depth reading about DC programming, the reader is referred to \cite{tao1997convex,horst1999dc}. Next, we define a DC function.

\begin{defi}[\cite{alamo_set2008}]\label{defi_dc_function}
Consider a polytope $\mathcal{P} \subset \mathbb{R}^n$ and a function $\varrho:\mathbb{R}^n \to \mathbb{R}^m$. If $\varrho$ can be rewritten as the difference between two convex functions $\varrho^{\rm a}$ and $\varrho^{\rm b}$ in $\mathcal{P}$, then, $\varrho$ is called DC on $\mathcal{P}$.
\end{defi}


The determination of the convex functions $\varrho^{\rm a}$ and $\varrho^{\rm b}$ may not be a trivial task. Some procedures to guide the choice of such functions are resumed in \cite{alamo_set2008}.
Next, we define the general form of a DC programming problem. After, two results are presented to enclose the global solutions of DC programming in intervals.


\begin{defi}[\cite{alamo_set2008}]\label{defi_dc_prog_problem}
Consider that the function $\varrho : \mathbb{R}^n \to \mathbb{R}^m$ is DC on $\mathcal{P} \subset \mathbb{R}^n$, with $\varrho^{\rm a}$ and $\varrho^{\rm b}$ being its DC components such that $\varrho(z) = \varrho^{\rm a}(z) - \varrho^{\rm b}(z)$.
Then, for each component $i = 1,\ldots,m$, the $i$th DC programming \AP{problems are} formulated as
\be
\label{eq_gen_prob_DC}
\min_{z \in \mathcal{P}} \varrho_i(z),~ \max_{z \in \mathcal{P}} \varrho_i(z).
\ee
\end{defi}

\begin{defi}[\cite{alamo_set2008}] \label{defi_linear_minorant}
Let $\varrho(z) = \left(\varrho^{\rm a}(z) - \varrho^{\rm b}(z)\right) \in \mathbb{R}^m$ be DC on $\mathcal{P} \subset \mathbb{R}^n$. Then, the linear minorant of $\varrho^{\rm s}$, with ${\rm s = \{a,b\}}$, is defined as
\be
\bar{\varrho}^{\rm s}(z) \triangleq \varrho^{\rm s}(\bar{z}) + F^{\rm s}(z - \bar{z}),
\ee
where
\be
F^{\rm s} \triangleq \nabla_z\varrho^{\rm s}(\bar{z}) = \bb 
\frac{\partial \varrho^{\rm s}_1}{\partial z_1} & \cdots & \frac{\partial \varrho^{\rm s}_1}{\partial z_n} \\
\vdots & \ddots & \vdots \\
\frac{\partial \varrho^{\rm s}_m}{\partial z_1} & \cdots & \frac{\partial \varrho^{\rm s}_m}{\partial z_n}
\eb \Bigg|_{\bar{z}}
\ee
is the Jacobian matrix 
evaluated at some $\bar{z} \in \mathcal{P}$. The term \emph{minorant} comes from the convexity of $\varrho^{\rm s}$ that implies the inequalities $\varrho^{\rm s}(z) \ge \bar{\varrho}^{\rm s}(z)$, $\forall z \in \mathcal{P}$.
\end{defi}


\begin{propo}[Adapted from \cite{alamo_set2008}]\label{propo_DC_prog_prob}
Let $\varrho(z) = \left(\varrho^{\rm a}(z) - \varrho^{\rm b}(z)\right) \in \mathbb{R}^m$ be a DC function on the polytope $\mathcal{P} \subset \mathbb{R}^n$. Then, according to Definition \ref{defi_linear_minorant}, the following inequalities hold:
\begin{align}
\label{eq_final_form_DC_min}
\min_{z \in \mathcal{P}} \varrho_i(z) &\ge \min_{z \in {\rm vert}(\mathcal{P})} \bar{\varrho}^{\rm a}_i(z) - \varrho^{\rm b}_i(z),\\
\label{eq_final_form_DC_max}
\max_{z \in \mathcal{P}} \varrho_i(z) &\le \max_{z \in {\rm vert}(\mathcal{P})} \varrho^{\rm a}_i(z) - \bar{\varrho}^{\rm b}_i(z),
\end{align}
for $i = 1,\ldots,m$, with ${\rm vert}(\mathcal{P})$ being the set of vertices of $\mathcal{P}$, with $\bar{\varrho}^{\rm a}_i(z) - \varrho^{\rm b}_i(z)$ being a concave function, and with $\varrho^{\rm a}_i(z) - \bar{\varrho}^{\rm b}_i(z)$ being a convex function.
\end{propo}



According to Proposition \ref{propo_DC_prog_prob}, the optimization problems are influenced by two factors, namely: (i) polytope $\mathcal{P}$; and (ii) convex functions $\varrho^{\rm a}$ and $\varrho^{\rm b}$ in $\mathcal{P}$. Regarding (i), the vertices of a polytope are individually evaluated over \eqref{eq_final_form_DC_min}-\eqref{eq_final_form_DC_max}. 
In order to control the number of vertices, and thereby, the computational cost associated to \eqref{eq_final_form_DC_min}-\eqref{eq_final_form_DC_max}, we outer approximate the polytope $\mathcal{P}$ by either its interval hull $\Box \mathcal{P} \supseteq \mathcal{P}$ or a parallelotope $\mathcal{T} \supseteq \mathcal{P}$, since these representations involve $2^n$ vertices. 
To obtain the vertices of both box $\Box \mathcal{P}$ and parallelotope $\mathcal{T}$, we can employ specific algorithms as those presented in \cite{dyer1983complexity}. 
Polytopes are here represented by CZs. Therefore, the interval hull $\Box \mathcal{P}$ is obtained with Proposition \ref{propo_int_hull_cz}. In order to obtain a tight parallelotope $\mathcal{T}$ to polytope $\mathcal{P}$, we next propose a new result, in which a candidate parallelotope $\mathcal{C} \supseteq \mathcal{P}$ is tightened via linear programs. 
The set $\mathcal{C}$ is here computed as in Assumption \ref{assu_cand_zon_from_cz}.

\begin{assu}\label{assu_cand_zon_from_cz}
Given a CZ $\mathcal{P}$, a parallelotope $\mathcal{C} \supseteq \mathcal{P}$ is obtained by reducing all constraints and generators of $\mathcal{P}$ with \AP{Method 4 of \cite{yang2018comparison}}.
\end{assu}

\begin{propo}\label{propo_partope}
Let $\mathcal{C} = \left\{G^{\rm c},c^{\rm c}\right\} \subset \mathbb{R}^n$ be a parallelotope containing the CZ $\mathcal{P} = \left\{G^{\rm z},c^{\rm z},A^{\rm z},b^{\rm z}\right\} \subset \mathbb{R}^n$.
By solving the linear programs
\begin{align*}
\zeta^{\rm L}_i &= \min_{\xi^{\rm c},\xi^{\rm z}} \left\{\xi^{\rm c}_i : G^{\rm c}_{i,:}\xi^{\rm c} + c^{\rm c}_i = G^{\rm z}_{i,:}\xi^{\rm z} + c^{\rm z}_i , \xi^{\rm c} \in \mathcal{B}^n , \xi^{\rm z} \in \mathcal{B}\left(A^{\rm z},b^{\rm z}\right)\right\},\\
\zeta^{\rm U}_i &= \max_{\xi^{\rm c},\xi^{\rm z}} \left\{\xi^{\rm c}_i : G^{\rm c}_{i,:}\xi^{\rm c} + c^{\rm c}_i = G^{\rm z}_{i,:}\xi^{\rm z} + c^{\rm z}_i , \xi^{\rm c} \in \mathcal{B}^n , \xi^{\rm z} \in \mathcal{B}\left(A^{\rm z},b^{\rm z}\right)\right\},
\end{align*}
for $i = 1,\ldots,n$, we obtain the optimal parallelotope $\mathcal{T} = \left\{G^{\rm c}{\rm diag}\left({\rm rad}\left(\left[\zeta^{\rm L},\zeta^{\rm U}\right]\right)\right) , c^{\rm c} + G^{\rm c}{\rm mid}\left(\left[\zeta^{\rm L},\zeta^{\rm U}\right]\right)\right\} \supseteq \mathcal{P}$.
\end{propo}

\begin{proof}
Given the candidate parallelotope $\mathcal{C} \supseteq \mathcal{P}$, a new parallelotope $\mathcal{T}$ is investigated to minimally contain the CZ $\mathcal{P}$. By fixing both the generator matrix $G^{\rm c}$ and the center $c^{\rm c}$ of $\mathcal{C}$, the slack variables to be manipulated are $\xi^{\rm c} \in \mathcal{B}^n$. Aiming at tightening the facets of $\mathcal{C}$ onto the CZ $\mathcal{P}$, $2n$ linear programs are formulated enforcing the $i$th linear equality constraint $G^{\rm c}_{i,:}\xi^{\rm c} + c^{\rm c}_i = G^{\rm z}_{i,:}\xi^{\rm z} + c^{\rm z}_i$, where $\xi^{\rm z} \in \mathcal{B}\left(A^{\rm z},b^{\rm z}\right)$, such that the minimization and maximization of $\xi^{\rm c}_i$ yield the smallest possible box $\left[\zeta^{\rm L},\zeta^{\rm U}\right] \subset \mathcal{B}^n$. Then, $\mathcal{T} = G^{\rm c}\left[\zeta^{\rm L},\zeta^{\rm U}\right] \oplus c^{\rm c}$ is the optimal parallelotopic outer approximation of $\mathcal{P}$, which can be rewritten as a zonotope using a rescaling \AP{(see \GR{$\mathcal{T}$} in the statement of Proposition \ref{propo_partope})}.
\end{proof}


The choice (ii) of the DC components is also important since the optimization problems \eqref{eq_final_form_DC_min}-\eqref{eq_final_form_DC_max} are other \GR{sources} of conservativeness. Note that different DC functions provide different bounds \cite{tao1997convex}. 
Having in mind this aspect, we present in Proposition \ref{propo_quad_DC_fun} a procedure to compute \AP{the DC decomposition of a function $\varrho$}. Specifically, we combine quadratic functions of \cite{alamo_set2008} with the choice of eigenvalue of \cite{adjiman1996rigorous}.
\AP{This choice \GR{of eigenvalue} is not unique, \GR{see for instance \cite{rohn1998bounds},} but it should be made carefully since large values imply conservative results in Proposition \ref{propo_DC_prog_prob}.}

\begin{propo}[Adapted from \cite{alamo_set2008,adjiman1996rigorous}]\label{propo_quad_DC_fun}
Let $\varrho: \mathbb{R}^n \to \mathbb{R}^m$ be a function of class $\mathcal{C}^2$ in $\mathcal{P}$, and $\Box \mathcal{P} \subset \mathbb{R}^n$ be the interval hull of $\mathcal{P}$. Consider functions $\varrho^{\rm a}_i(z) = \varrho_i(z) + \varrho^{\rm b}_i(z)$ and ${\displaystyle \varrho^{\rm b}_i(z) = \frac{\tilde{\lambda}_i}{2}z^{\top}z}$, for $i = 1,\ldots,m$, where
\be
\tilde{\lambda}_i = \max \left\{0,-\breve{\lambda}_i\right\},
\ee
with $\breve{\lambda}_i \in \mathbb{R}$ being computed as in \cite[Equation (12)]{adjiman1996rigorous}, which is a lower bound for the smallest eigenvalue of interval Hessian matrix $[H_i] = (\partial^2/\partial z^2) \varrho_i\left(\Box \mathcal{P}\right)$. Then, $\varrho = \varrho^{\rm a} - \varrho^{\rm b}$ is a DC function on $\Box \mathcal{P}$.
\end{propo}

\begin{rema}\label{rema_improved_quad_function}
Each function $\varrho^{\rm b}_i$ could be defined with $\frac{1}{2}z^{\top}Q_iz$, where $Q_i$ is a \AP{diagonal matrix whose elements \GR{could} be obtained by generalizing Proposition \ref{propo_quad_DC_fun} via semidefinite programming.}
\AP{
Moreover, instead of convexifying $\varrho$ (obtaining $\varrho^{\rm a}$), we could convexify $-\varrho$ and place this result in $\varrho^{\rm b}$. 
}
\end{rema}

%% file: 04CZDC_version01b.tex
\section{The Novel State Estimator}\label{sec_CZDC}

Next, we present the novel set-membership filter based on CZs and DC programming, called CZDC. This algorithm solves the problem formulated in Section \ref{sec_problem} in five steps. The general idea is to firstly linearize models. Then, operations \eqref{eq_aff_transf_cz}-\eqref{eq_intersection_cz} are performed on the linearized models. Finally, DC programming is employed to bound the linearization error. 

In order to obtain linearization error enclosures $\mathcal{R}$, we present Lemma \ref{lemm_lin_remainder_DC}. For practical reasons, the input CZ $\mathcal{Z}$ may be outer approximated by either a box $\Box \mathcal{Z}$ (Proposition \ref{propo_int_hull_cz}) or a parallelotope $\mathcal{T}$ (Proposition \ref{propo_partope}), yielding the desired polytope $\mathcal{P}$, before solving the problems \eqref{eq_DC_solve_fa}-\eqref{eq_DC_solve_fb}.

\begin{lemm}[Adapted from \cite{alamo_set2008}]\label{lemm_lin_remainder_DC}
Let $\varrho: \mathbb{R}^n \to \mathbb{R}^m$ be a function 
in the CZ $\mathcal{Z}$ whose first-order expansion is given by
\be
\label{eq_linearization_varrho}
\bar{\varrho}(z) = \varrho\left(\bar{z}\right) + F\left(z - \bar{z}\right),
\ee
where $F = \nabla_z\varrho(\bar{z})$, $\bar{z} \in \mathcal{P}$ is any punctual estimate, and $\mathcal{P} \supseteq \mathcal{Z}$ is a convex polytope \GR{(either a box or a parallelotope for computational reasons)}. Let $e(z) \triangleq \varrho(z) - \bar{\varrho}(z)$ be the linearization error. Let also $\varrho^{\rm a}$ and $\varrho^{\rm b}$ be convex functions such that $\varrho = \varrho^{\rm a} - \varrho^{\rm b}$ is DC on $\mathcal{P}$. Finally, let $\left(\varrho^{\rm a} - \bar{\varrho}^{\rm b} - \bar{\varrho}\right)$ be a convex majorant of $e$, and let $\left(\bar{\varrho}^{\rm a} - \varrho^{\rm b} - \bar{\varrho}\right)$ be a concave minorant of $e$. 
Then, according to Proposition \ref{propo_DC_prog_prob}, a linearization enclosure $\mathcal{R} = \left[e^-,e^+\right] \ni e$ is given by
\begin{align}
\label{eq_DC_solve_fa}
e^-_i &= \min_{z \in {\rm vert}\left(\mathcal{P}\right)} \left(\bar{\varrho}^{\rm a}_i\left(z\right) - \varrho^{\rm b}_i\left(z\right) - \bar{\varrho}_i\left(z\right) \right),\\
\label{eq_DC_solve_fb}
e^+_i &= \max_{z \in {\rm vert}\left(\mathcal{P}\right)} \left(\varrho^{\rm a}_i\left(z\right) - \bar{\varrho}^{\rm b}_i\left(z\right) - \bar{\varrho}_i\left(z\right) \right),
\end{align}
for $i = 1,\ldots,m$. Once \eqref{eq_DC_solve_fa}-\eqref{eq_DC_solve_fb} are solved, the intervals are expressed in affine arithmetic as the zonotope $\mathcal{R} = \left\{G^{\rm e},c^{\rm e}\right\}$ with
\begin{align}
\label{eq_Gzx_ef}
G^{\rm e} &= {\rm diag}\left({\rm rad}\left([e^-,e^+]\right)\right),\\
\label{eq_hat_ef}
c^{\rm e} &= {\rm mid}\left([e^-,e^+]\right).
\end{align}
\end{lemm}

\begin{proof}
This proof is similar to \cite[Proof of Lemma 1]{alamo_set2008}, with the difference being that $\mathcal{Z}$ is a CZ (instead of zonotope), 
and $\varrho$ is any function 
in $\mathcal{P}$.
\end{proof}



In the following, the results to execute a loop of CZDC are presented. We emphasize that Theorem \ref{theo_forecast} extends \cite[Theorem 1]{alamo_set2008} by introducing computations with CZs and deterministic input vector.

\begin{theo}[Forecast Step]\label{theo_forecast}
Consider the CZs $\mathcal{X}_{k-1} \subset \mathbb{R}^n$ and $\mathcal{W}_{k-1} \subset \mathbb{R}^q$, and the deterministic input $u_{k-1} \in \mathbb{R}^p$. Let $f:\mathbb{R}^n \times \mathbb{R}^p \times \mathbb{R}^q \to \mathbb{R}^n$ \eqref{eq_process_model} be rewritten as $\varrho^{\rm f}: \mathbb{R}^{(n+p+q)} \to \mathbb{R}^n$ using the augmented vector $z_{k-1} = \bb x^{\top}_{k-1} & u^{\top}_{k-1} & w^{\top}_{k-1} \eb^{\top}$. Let also $\varrho^{\rm f} = \varrho^{\rm fa} - \varrho^{\rm fb}$ be DC on the polytope $\mathcal{P}_{k-1} \supseteq \mathcal{Z}_{k-1} = \mathcal{X}_{k-1} \times u_{k-1} \times \mathcal{W}_{k-1}$. Finally, let $\mathcal{R}_{k-1} \subset \mathbb{R}^n$ be the set returned by Lemma \ref{lemm_lin_remainder_DC} to compensate the linearization error of $\varrho^{\rm f}$ for a given punctual estimate $\bar{z} \in \mathcal{P}_{k-1}$. Then, the exact image $\varrho^{\rm f}\left(\mathcal{Z}_{k-1}\right)$ is outer approximated by the CZ
\be
\mathcal{X}_{k|k-1} = \left(\varrho^{\rm f}\left(\bar{z}\right) - F^{\rm x}\bar{x} - F^{\rm w}\bar{w} \right) \oplus F^{\rm x}\mathcal{X}_{k-1} \oplus F^{\rm w}\mathcal{W}_{k-1} \oplus \mathcal{R}_{k-1},
\ee
with $F^{\rm x} = \nabla_x\varrho^{\rm f}\left(\bar{z}\right)$ and $F^{\rm w} = \nabla_w\varrho^{\rm f}\left(\bar{z}\right)$ being Jacobian matrices evaluated at $\bar{z} = \bb \bar{x}^{\top} & u^{\top}_{k-1} & \bar{w}^{\top} \eb^{\top}$.
\end{theo}

\begin{proof}
This proof is similar to \cite[Proof of Theorem 1]{alamo_set2008}, with the difference being the propagation of CZs instead of zonotopes.
\end{proof}


\begin{theo}[Data-Assimilation Step]\label{theo_data_ass}
Consider the CZs $\mathcal{X}_{k|k-1} \subset \mathbb{R}^n$ and $\mathcal{V}_k \subset \mathbb{R}^r$, and  the measured output $y_k \in \mathbb{R}^m$. Let $h:\mathbb{R}^n \times \mathbb{R}^r \to \mathbb{R}^m$ \eqref{eq_measurement_model} be rewritten as $\varrho^{\rm h}: \mathbb{R}^{(n+r)} \to \mathbb{R}^m$ using the augmented vector $z_k = \bb x^{\top}_k & v^{\top}_k \eb^{\top}$. Let also $\varrho^{\rm h} = \varrho^{\rm ha} - \varrho^{\rm hb}$ be DC on the polytope $\mathcal{P}_k \supseteq \mathcal{Z}_k = \mathcal{X}_{k|k-1} \times \mathcal{V}_k$. Finally, let $\mathcal{R}_k \subset \mathbb{R}^m$ be the set returned by Lemma \ref{lemm_lin_remainder_DC} to compensate the linearization error of $\varrho^{\rm h}$ for a given punctual estimate $\bar{z} \in \mathcal{P}_k$. Then, the exact set $\left\{x_k \in \mathcal{X}_{k|k-1}: y_k = h\left(x_k,v_k\right),v_k \in \mathcal{V}_k\right\}$ is over approximated by the CZ
\be
\breve{\mathcal{X}}_k = \mathcal{X}_{k|k-1} \cap_{H^{\rm x}} \mathcal{Y}_k,
\ee
where 
$\mathcal{Y}_k = \left(y_k - \varrho^{\rm h}(\bar{z}) + H\bar{z}\right) \oplus \left(-H^{\rm v}\mathcal{V}_k\right) \oplus \left(-\mathcal{R}_k\right)$, with $H = \bb H^{\rm x} & H^{\rm v} \eb$, $H^{\rm x} = \nabla_x\varrho^{\rm h}(\bar{z})$, and $H^{\rm v} = \nabla_v\varrho^{\rm h}(\bar{z})$.
\end{theo}

\begin{proof}
Let
\bes
y_k = \varrho^{\rm h}(\bar{z}) + H\left(z_k - \bar{z}\right) + e^{\rm h}_k
\ees
be the analytical linearization of the DC function $\varrho^{\rm h} = \varrho^{\rm ha} - \varrho^{\rm hb}$ on $\mathcal{P}_k$, and let the CZ $\mathcal{R}_k \ni e^{\rm h}_k$ be the linearization error enclosure given by Lemma \ref{lemm_lin_remainder_DC}.
By making explicit the term $H^{\rm x}x_k$ from $Hz_k = \bb H^{\rm x} & H^{\rm v} \eb \bb x_k \\ v_k \eb$, we obtain $H^{\rm x}x_k = y_k - \varrho^{\rm h}(\bar{z}) + H\bar{z} - H^{\rm v}v_k - e^{\rm h}_k$, which implies the CZ
$\mathcal{Y}_k = \left(y_k - \varrho^{\rm h}(\bar{z}) + H\bar{z}\right) \oplus \left(-H^{\rm v}\mathcal{V}_k\right) \oplus \left(-\mathcal{R}_k\right)$.
Then, we employ the generalized intersection \eqref{eq_intersection_cz} to match $\mathcal{X}_{k|k-1}$ with $\mathcal{Y}_k$, yielding $\breve{\mathcal{X}}_k$.
\end{proof}


\begin{rema}
If functions $f$ and $h$ are affine in the noise terms $w_{k-1}$ and $v_k$, respectively, then these terms are canceled during the computation of $\mathcal{R}$ in Theorems \ref{theo_forecast} and \ref{theo_data_ass}. It means that, instead of $2^{(n+q)}$ and $2^{(n+r)}$ vertices, we need to process $2^n$ vertices only.
\end{rema}

Since the consistency step is a direct consequence from Theorem \ref{theo_data_ass}, it is next presented as a corollary.

\begin{coro}[Consistency Step]\label{coro_consistency}
Consider the CZ $\breve{\mathcal{X}}_k \subset \mathbb{R}^n$ (Theorem \ref{theo_data_ass}) and the feasible set $\mathcal{X}^{\rm F}_k \subset \mathbb{R}^n$. Let $\check{\mathcal{X}}_k = \breve{\mathcal{X}}_k \cap_{{\rm I}_n} \mathcal{X}^{\rm F}_k$ be the admissible set (admissibility step in Section \ref{sec_problem}). Let $g:\mathbb{R}^n \to \mathbb{R}^{m_c}$ \eqref{eq_equality_constraints} be rewritten as $g = g^{\rm a} - g^{\rm b}$, where $g^{\rm a}$ and $g^{\rm b}$ are convex functions in the polytope $\check{\mathcal{P}}_k \supseteq \check{\mathcal{X}}_k$. Let also $\mathcal{R}_k \subset \mathbb{R}^{m_c}$ be the set returned by Lemma \ref{lemm_lin_remainder_DC} to compensate the linearization error of $g$ for a given punctual estimate $\bar{x} \in \check{\mathcal{P}}_k$. Then, the exact set $\left\{x_k \in \check{\mathcal{X}}_k: g\left(x_k\right) = 0_{m_c \times 1}\right\}$ is over approximated by the CZ 
\be
\tilde{\mathcal{X}}_k = \check{\mathcal{X}}_k \cap_{H} \mathcal{C}_k,
\ee
where $H = \nabla_xg(\bar{x})$ and 
$\mathcal{C}_k = \left(-g(\bar{x}) + H\bar{x}\right) \oplus \left(-\mathcal{R}_k\right).$
\end{coro}

\begin{proof}
This proof is similar to the proof of Theorem \ref{theo_data_ass}, whose difference is the replacement of $y_k$, $h$, and $z_k$ by $0_{m_c \times 1}$, $g$, and $x_k$, respectively.
\end{proof}




We summarize the steps of CZDC in Algorithm \ref{algo_CZDC}. 

\begin{algorithm}[!ht]
\SetAlgoLined

    \caption{$\mathcal{X}_k = {\rm CZDC} \big(f,f^{\rm a},f^{\rm b},\mathcal{X}_{k-1},u_{k-1},\mathcal{W}_{k-1},$
    $y_k,h,h^{\rm a},h^{\rm b},\mathcal{V}_k,g,g^{\rm a},g^{\rm b},\mathcal{X}^{\rm F}_k,\varphi_c,\varphi_g \big)$}
   
    \label{algo_CZDC}
    
    Apply Theorem \ref{theo_forecast} to obtain the CZ $\mathcal{X}_{k|k-1}$
    \\
    Apply Theorem \ref{theo_data_ass} to obtain the CZ $\breve{\mathcal{X}}_k$
    \\
    Compute $\check{\mathcal{X}}_k = \breve{\mathcal{X}}_k \cap_{{\rm I}_n} \mathcal{X}^{\rm F}_k$
    \\
    Apply Corollary \ref{coro_consistency} to obtain the CZ $\tilde{\mathcal{X}}_k$
    \\
    Apply the algorithm proposed in \cite{scott2016} to reduce the number of constraints $n_h$ and generators $n_g$ of $\tilde{\mathcal{X}}_k$ to $\varphi_c$ and $\varphi_g$, respectively, yielding the CZ $\mathcal{X}_k$
    
\end{algorithm}

\subsection{Complexity Analysis}

The worst-case computational complexity $O(\cdot)$ for each step of CZDC (Algorithm \ref{algo_CZDC}) is shown in Table \ref{tab_complex_analysis_ref}. Such complexities were derived using basic operations among CZs \cite{rego2020guaranteed}.
Regarding the forecast, data assimilation, and consistency steps, the complexity order to obtain the linearization point $\bar{z}$ is not included since it depends on the employed methodology. As in \cite{rego2020guaranteed}, we also assume that the evaluation of nonlinear functions has complexity $O(1)$.
In the second column of Table \ref{tab_complex_analysis_ref}, the cubic term between parenthesis refers to either the computation of parallelotope via linear programs or to the Hausdorff distance minimization (order reduction). The term $2^{\tilde{n}}$ is related to either the computation of vertices or the DC programming problems \eqref{eq_DC_solve_fa}-\eqref{eq_DC_solve_fb}.
In turn, the third column of Table \ref{tab_complex_analysis_ref} presents the amount of constraints and generators for the state CZ $\mathcal{X}$ over the different steps.

Table \ref{tab_complex_analysis_ref} makes the following assumptions: $\mathcal{X}_{k-1} = \left\{G^{\rm x}_{k-1},c^{\rm x}_{k-1},A^{\rm x}_{k-1},b^{\rm x}_{k-1}\right\} \subset \mathbb{R}^n$, $\mathcal{W}_{k-1} = \left\{G^{\rm w}_{k-1},c^{\rm w}_{k-1},A^{\rm w}_{k-1},b^{\rm w}_{k-1}\right\} \subset \mathbb{R}^q$, $\mathcal{V}_k = \left\{G^{\rm v}_k,c^{\rm v}_k,A^{\rm v}_k,b^{\rm v}_k\right\} \subset \mathbb{R}^r$, and $\mathcal{X}^{\rm F}_k = \left\{G^{\rm x^F}_k,c^{\rm x^F}_k,A^{\rm x^F}_k,b^{\rm x^F}_k\right\} \subset \mathbb{R}^n$, where $G^{\rm x}_{k-1} \in \mathbb{R}^{n \times n_g}$, $G^{\rm w}_{k-1} \in \mathbb{R}^{q \times n^{\rm w}_g}$, $G^{\rm v}_k \in \mathbb{R}^{r \times n^{\rm v}_g}$, $G^{\rm x^F}_k \in \mathbb{R}^{n \times n^{\rm x^F}_g}$, $c^{\rm x}_{k-1} \in \mathbb{R}^n$, $c^{\rm w}_{k-1} \in \mathbb{R}^q$, $c^{\rm v}_k \in \mathbb{R}^r$, $c^{\rm x^F}_k \in \mathbb{R}^n$, $A^{\rm x}_{k-1} \in \mathbb{R}^{n_h \times n_g}$, $A^{\rm w}_{k-1} \in \mathbb{R}^{n^{\rm w}_h \times n^{\rm w}_g}$, $A^{\rm v}_k \in \mathbb{R}^{n^{\rm v}_h \times n^{\rm v}_g}$, $A^{\rm x^F}_k \in \mathbb{R}^{n^{\rm x^F}_h \times n^{\rm x^F}_g}$, $b^{\rm x}_{k-1} \in \mathbb{R}^{n_h}$, $b^{\rm w}_{k-1} \in \mathbb{R}^{n^{\rm w}_h}$, $b^{\rm v}_k \in \mathbb{R}^{n^{\rm v}_h}$, and $b^{\rm x^F}_k \in \mathbb{R}^{n^{\rm x^F}_h}$.
These sets are evaluated over the functions $f:\mathbb{R}^n \times \mathbb{R}^p \times \mathbb{R}^q \to \mathbb{R}^n$, $h:\mathbb{R}^n \times \mathbb{R}^r \to \mathbb{R}^m$, and $g: \mathbb{R}^n \to \mathbb{R}^{m_c}$, considering the vectors $u_{k-1} \in \mathbb{R}^p$ and $y_k \in \mathbb{R}^m$.
At the end of an iteration of CZDC, the desired CZ $\mathcal{X}_k$ is returned with $\varphi_c$ constraints and $\varphi_g$ generators.


\begin{rema}\label{rema_order_comparison}
According to Table \ref{tab_complex_analysis_ref}, the output CZs obtained by Theorems \ref{theo_forecast}-\ref{theo_data_ass} and Corollary \ref{coro_consistency} have smaller number of constraints and generators than those pointed out by Remarks 5, 7, and 10 from \cite{rego2021set}. Exceptionally, the number of constraints for $\mathcal{X}_{k|k-1}$ coincides with the value indicated in \cite[Remark 5]{rego2021set} for the CZMV algorithm.
\end{rema}

\begin{table*}[t]\small
	\centering
	\caption{Complexity order of the forecast, data assimilation, admissibility, consistency, and reduction steps from CZDC using Proposition \ref{propo_partope}.}
	\label{tab_complex_analysis_ref}
	\begin{tabular}{l|c|c}
	\hline
	Step & $O(\cdot)$ & Definition \\ \hline
    Forecast & $\tilde{n}_h\left(\tilde{n}_h + \tilde{n}_g\right)^3 + \tilde{n}\left(\tilde{n} + \tilde{n}_g\right)\left(\tilde{n} + \tilde{n}_h + \tilde{n}_g\right)^3 + \tilde{n}^22^{\tilde{n}}$ & $\tilde{n} = n + q$, $\tilde{n}_h = n_h + n^{\rm w}_h$, $\tilde{n}_g = n_g + n^{\rm w}_g$
    \\ \hline
    \multirow{2}{*}{Data Assimilation} & $\tilde{n}_h\left(\tilde{n}_h + \tilde{n}_g\right)^3 + \tilde{n}\left(\tilde{n} + \tilde{n}_g\right)\left(\tilde{n} + \tilde{n}_h + \tilde{n}_g\right)^3$ & \multirow{2}{*}{$\tilde{n} = n + r$, $\tilde{n}_h = n_h + n^{\rm w}_h + n^{\rm v}_h$, $\tilde{n}_g = n_g + n^{\rm w}_g + n + n^{\rm v}_g$} \\ & $+ \left(\tilde{n}^2 + m\right)2^{\tilde{n}} + mrn^{\rm v}_g + mn\left(\tilde{n}_g - n^{\rm v}_g\right) + m\tilde{n} + m^2$ &
    \\ \hline
    Admissibility & $n^2\tilde{n}_g + nn^{\rm x^F}_g$ & $\tilde{n}_g = n_g + n^{\rm w}_g + n + n^{\rm v}_g + m$
    \\ \hline
    \multirow{2}{*}{Consistency} & \multirow{2}{*}{$\tilde{n}_h\left(\tilde{n}_h + \tilde{n}_g\right)^3 + n\left(n + \tilde{n}_g\right)\left(n + \tilde{n}_h + \tilde{n}_g\right)^3 + \left(n^2 + m_c\right)2^n + m_cn\tilde{n}_g + m^2_c$} & $\tilde{n}_h = n_h + n^{\rm w}_h + n^{\rm v}_h + m + n^{\rm x^F}_h + n$ \\
    & & $\tilde{n}_g = n_g + n^{\rm w}_g + n +  n^{\rm v}_g + m + n^{\rm x^F}_g$
    \\ \hline
    \multirow{3}{*}{Reduction} & \multirow{3}{*}{$k_c\left(\tilde{n}_h + \tilde{n}_g\right)^3 + k_cn\tilde{n}^2_g + \left(n + \varphi_c\right)^2(\tilde{n}_g - k_c) + k_g(\tilde{n}_g - k_c)\left(n + \varphi_c\right)$} & $k_c = \tilde{n}_h - \varphi_c$, $k_g = \tilde{n}_g - k_c - \varphi_g$, \\ 
    & & $\tilde{n}_h = n_h + n^{\rm w}_h + n^{\rm v}_h + m + n^{\rm x^F}_h + n + m_c$, \\
    & & $\tilde{n}_g = n_g + n^{\rm w}_g + n +  n^{\rm v}_g + m + n^{\rm x^F}_g + m_c$ \\ \hline
	\end{tabular}
\end{table*}

%% file: 05NumericalExamples_version01b.tex
\section{Numerical Results}\label{sec_numerical_examples}

In this section, CZDC is experimented over two case studies. For comparison purposes, we also implement the state-of-the-art algorithms proposed in \cite{rego2021set}, called CZFO (based on Taylor expansion) and CZMV (based on mean value extension). 
To yield punctual estimates $\bar{z}$, and thereby, to approximate the nonlinear models, we make the following choices: 
CZDC uses the center of the polytope $\mathcal{P}$ associated to Lemma \ref{lemm_lin_remainder_DC}, \AP{where $\mathcal{P}$ is a box or a parallelotope}, whose procedure is $O(1)$; 
CZFO is run with metric \cite[$C3$]{rego2021set} to minimize the diameter of an interval matrix; 
CZMV is run with metric \cite[$C2$]{rego2021set} to minimize the diameter of an interval vector.
Two performance indexes are computed, namely: (i) the {\em mean processing time} ($T^{\rm CPU}$), given by $\displaystyle T^{\rm CPU} \triangleq \frac{1}{m_s}\frac{1}{k_f}\sum_{j=1}^{m_s}\sum_{k=1}^{k_f}t_{k,j}$, where $k_f \in \mathbb{N}$ is the number of time steps, $m_s \in \mathbb{N}$ is the number of Monte Carlo simulations, and $t_{k,j}$ is the time to execute the $k$th iteration of a given algorithm in the $j$th Monte Carlo simulation; and (ii) 
\AP{the \emph{average area ratio of box} ($A^{\Box}$), given by $$\displaystyle A^{\Box} \triangleq \frac{1}{m_s}\frac{1}{k_f}\sum_{j=1}^{m_s}\sum_{k=1}^{k_f} \prod^{n}_{i=1} {\rm diam}\left([x]_{i,k,j}\right),$$
with ${\rm diam}([x]) = 2{\rm rad}([x])$.}
The noise terms $w_{k-1}$ and $v_k$ are taken from uniform distributions defined in $\mathcal{W}_{k-1}$ and $\mathcal{V}_k$\AP{, while the initial state $x_0$ belongs to the initial set $\mathcal{X}_0$.}
The following computer configuration was used: 8 GB RAM 1333 MHz, Windows 10 Pro, and AMD FX-6300 CPU 3.50 GHz. All implementations were executed in MATLAB 9.11 with INTLAB 12 \cite{Ru99a}, MPT3 \cite{MPT3}, and Gurobi 9.1.

Since the measurement $y_0$ is available, all three algorithms execute a first loop with $\mathcal{X}_{k|k-1} = \mathcal{X}_0$, whose goal is to improve the precision of the starting set $\mathcal{X}_0$. Soon after, the state estimators are normally executed. For all examples, CZFO employs order reduction with fixed values $\varphi_c$ and $\varphi_g$ at the end of each step, as recommended in \cite{rego2021set}. 

\subsection{Two-State Nonlinear Process}\label{subsec_first_exam}

Consider the nonlinear uncertain system \cite{alamo_set2008}
\begin{align}
\label{eq_process_alamo2008}
x_k &= \bb -0.7x_{2,k-1} + 0.1x^2_{2,k-1} + 0.1x_{1,k-1}x_{2,k-1} + 0.1\exp\left(x_{1,k-1}\right) \\ x_{1,k-1} + x_{2,k-1} - 0.1x^2_{1,k-1} + 0.2x_{1,k-1}x_{2,k-1} \eb \nonumber \\
&+ w_{k-1},\\
y_k &= x_{1,k} + x_{2,k} + v_k,
\end{align}
where $w_{k-1} \in \mathcal{W} = \left\{0.1{\rm I}_2,0_{2 \times 1}\right\}$ and $v_k \in \mathcal{V} = \left\{0.2,0\right\}$. To simulate this system, we set $x_0 = \bb 1 & 1\eb^{\top} \in \mathcal{X}_0 = \left\{3 \times {\rm I}_2,0_{2 \times 1}\right\}$, $k_f = 40$, and $m_s = 100$.
This example aims at illustrating that CZDC is a promising option to substitute the use of CZFO and CZMV whenever the wrapping and dependency effects imply divergence of estimates, and that CZDC reaches a better precision than the zonotopic filter based on DC programming (ZDC) proposed in \cite{alamo_set2008}. 
\AP{To reduce order of CZs, we set $\varphi_c = 3$ and $\varphi_g = 8$. This latter value is also used to reduce order of zonotopes in ZDC with Method 4 of \cite{yang2018comparison}.}
To improve both the computational efficiency and the precision of the minimum-volume zonotopes computed in ZDC, we employ \cite[VM3]{alesi2022CSM} and \cite[Definition 8]{bravo_bounded2006}. 
Motivated by \cite{alamo_set2008}, we propose the DC function $x_k = f^{\rm a} - f^{\rm b} + w_{k-1}$ such that
\begin{align*}
f^{\rm a} &= \bb 0.1x^2_{1,k-1} + 0.1x_{1,k-1}x_{2,k-1} + 0.1x^2_{2,k-1} + \AP{0.1\exp\left(x_{1,k-1}\right)} \\  0.1x^2_{2,k-1} + x_{1,k-1} + x_{2,k-1} \eb,\\
f^{\rm b} &= \bb 0.1x^2_{1,k-1} + 0.7x_{2,k-1} \\ 0.1x^2_{1,k-1} + 0.1x^2_{2,k-1} - 0.2x_{1,k-1}x_{2,k-1} \eb.
\end{align*}
Since DC functions were directly defined, \AP{Proposition \ref{propo_quad_DC_fun} was not employed, and thereby, the polytope $\mathcal{P}_{k-1}$ associated to Theorem \ref{theo_forecast} is a parallelotope (given by Proposition \ref{propo_partope})} that contains the CZ $\mathcal{Z}_{k-1} = \mathcal{X}_{k-1}$.

In Figure \ref{fig_estimation_alamo2008_ref}(a), we point out that both CZFO and CZMV diverge due to the direct usage of interval arithmetic. Although this interval extension was used to experiment CZFO and CZMV in \cite{rego2021set}, it is not enough to reach convergence in this case study. 

Differently, both ZDC and CZDC achieve convergent solutions because DC programming involves evaluation of elementary functions rather than inclusion functions. In Figure \ref{fig_estimation_alamo2008_ref}(b) and (c), one-dimensional intervals are sketched to illustrate that those algorithms provide guaranteed solutions. As shown in Table \ref{tab_alamo2008_ref}, CZDC provides a significantly better precision than ZDC at the cost of a larger $T^{\rm CPU}$.

\begin{figure}[!tb]
	\center
	\def\svgwidth{8cm}
	{\scriptsize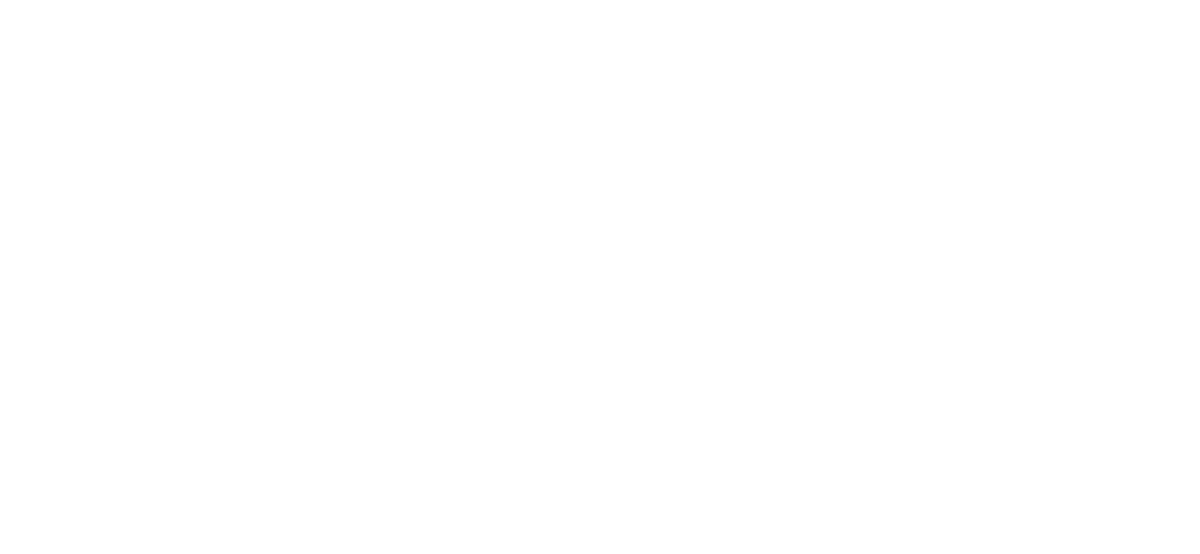}
	\\
	(a)\\
	\def\svgwidth{8cm}
	{\scriptsize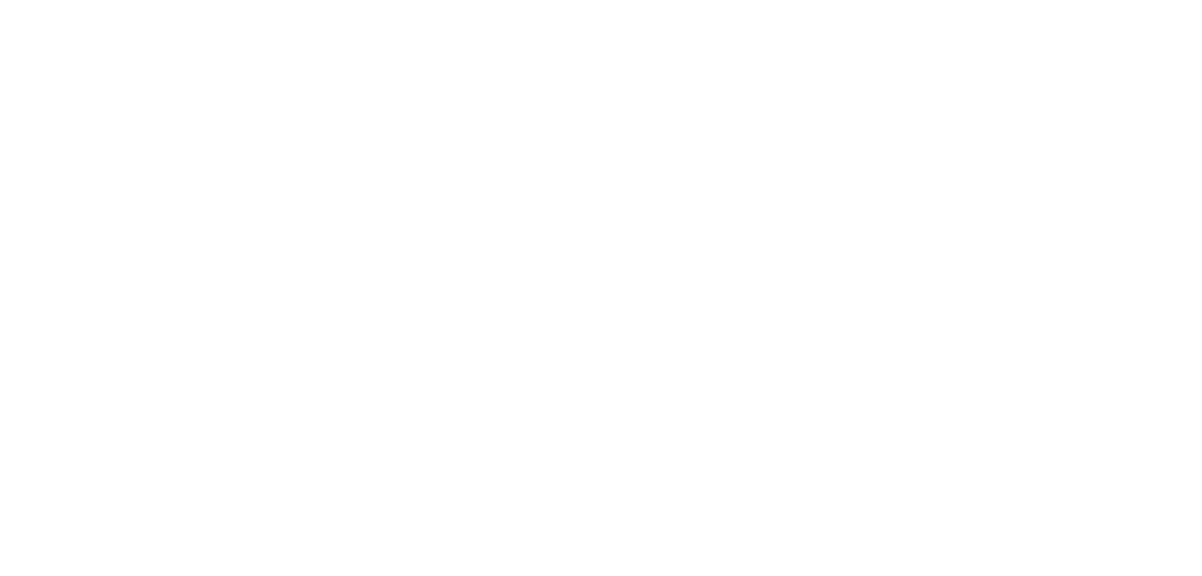}\\
	(b)\\
	\def\svgwidth{8cm}
	{\scriptsize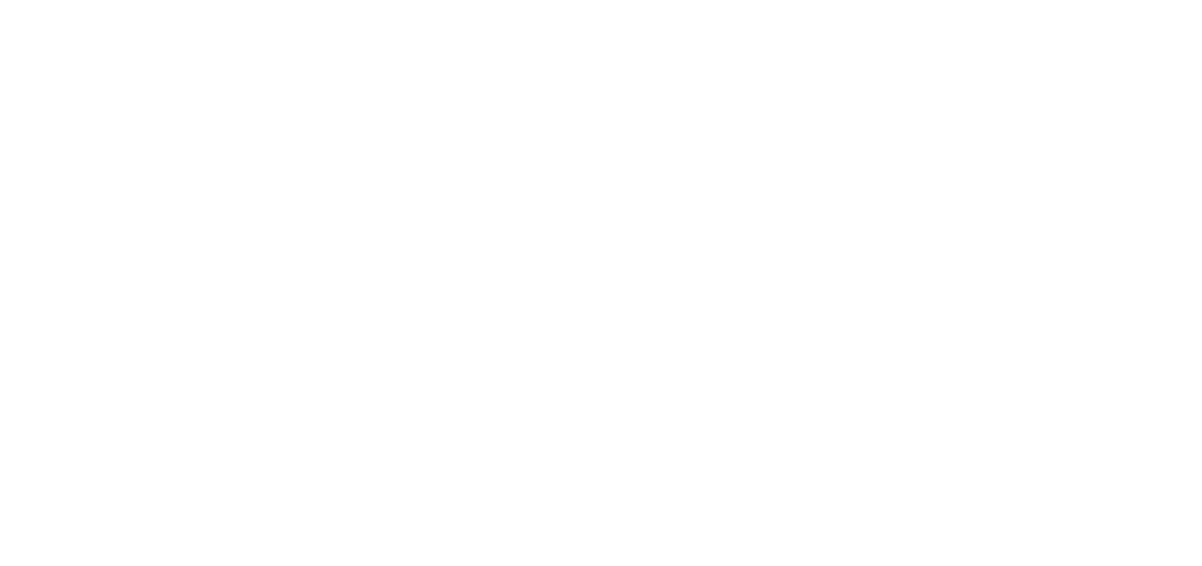}\\
	(c)
	\caption{State estimation for the first case study (Subsection \ref{subsec_first_exam}). Graph (a) depicts the time evolution of the \AP{area} of boxes computed by ZDC, CZDC, CZFO, and CZMV. In (b) and (c), true states are involved by interval hulls of CZs computed by ZDC and CZDC.}
	\label{fig_estimation_alamo2008_ref}
\end{figure}

\begin{table}[t]
	\centering
	\caption{Results of $T^{\rm CPU}$ and \AP{$A^{\Box}$} for the first example (Subsection \ref{subsec_first_exam}).}
	\label{tab_alamo2008_ref}
	\begin{tabular}{c|c|c}
	\hline
    Indexes & ZDC & CZDC \\ \hline
    $T^{\rm CPU}$ & {\bf 7.90 ms} & 12.6 ms \AP{($\uparrow$59.5\%)} \\
	\AP{$A^{\Box}$} & 3.62 & {\bf 1.75} \AP{($\downarrow$51.7\%)} \\ \hline
	\end{tabular}
\end{table}

\subsection{Attitude Estimation}\label{subsec_second_exam}

Now, we show the application of CZDC to a more challenging and technological example, containing multiplicative process noise, nonlinear measurements, and state equality  constraints. The considered system concerns the attitude estimation of a flying robot. By employing quaternion representation, the attitude is expressed as $x_k \in \mathbb{R}^4$ such that $\left|\left|x_k\right|\right|^2_2 = 1$. These states evolve at discrete time according to \cite{rego2021set,teixeira2009state}
\be
\label{eq_process_attitude}
x_k = \bigg(\cos\left(p\left(\check{u}_{k-1}\right)\right){\rm I}_4 - \frac{T_s}{2}\frac{\sin\left(p\left(\check{u}_{k-1}\right)\right)}{p\left(\check{u}_{k-1}\right)}\Omega\left(\check{u}_{k-1}\right)\bigg)x_{k-1},
\ee
where $T_s = 0.2~ {\rm s}$ is the sampling time, $p\left(\check{u}_k\right) = \frac{T_s}{2}\left|\left|\check{u}_k\right|\right|_2$, $\Omega\left(\check{u}_k\right) = \bb 0 & \check{u}_{3,k} & -\check{u}_{2,k} & \check{u}_{1,k} \\ -\check{u}_{3,k} & 0 & \check{u}_{1,k} & \check{u}_{2,k} \\ \check{u}_{2,k} & -\check{u}_{1,k} & 0 & \check{u}_{3,k} \\ -\check{u}_{1,k} & -\check{u}_{2,k} & -\check{u}_{3,k} & 0 \eb$, and $\check{u}_k = \bb 0.3\sin\left(\left(2\pi/12\right)kT_s\right) \\ 0.3\sin\left(\left(2\pi/12\right)kT_s - 6\right) \\ 0.3\sin\left(\left(2\pi/12\right)kT_s - 12\right) \eb$ is the physical input that drives the actual system. For state-estimation purposes, we assume that $\check{u}_k$ is acquired by gyroscopes. Then, $\check{u}_k$ is corrupted by an additive noise $w_k \in \mathcal{W} = \left\{3 \times 10^{-3}{\rm I}_3,0_{3 \times 1}\right\}$, whose result is the known signal $u_k = \check{u}_k + w_k$.
The measurement is given by
\be
\label{eq_measurement_attitude}
y_k = \bb C\left(x_k\right)r^{[1]} \\ C\left(x_k\right)r^{[2]} \eb + v_k,
\ee
where $r^{[1]} = \bb 1 & 0 & 0 \eb^{\top}$, $r^{[2]} = \bb 0 & 1 & 0 \eb^{\top}$,
\bes
C\left(x_k\right) =
\!\begin{aligned}
&
\left[\begin{matrix}
x^2_{1,k} - x^2_{2,k} - x^2_{3,k} + x^2_{4,k} & 2\left(x_{1,k}x_{2,k} + x_{3,k}x_{4,k}\right) \\
2\left(x_{1,k}x_{2,k} - x_{3,k}x_{4,k}\right) & -x^2_{1,k} + x^2_{2,k} - x^2_{3,k} + x^2_{4,k} \\
2\left(x_{1,k}x_{3,k} + x_{2,k}x_{4,k}\right) & 2\left(-x_{1,k}x_{4,k} + x_{2,k}x_{3,k}\right)
\end{matrix}\right.\\
&\qquad \qquad \qquad \qquad \qquad \quad
\left.\begin{matrix}
2\left(x_{1,k}x_{3,k} - x_{2,k}x_{4,k}\right) \\
2\left(x_{1,k}x_{4,k} + x_{2,k}x_{3,k}\right) \\
-x^2_{1,k} - x^2_{2,k} + x^2_{3,k} + x^2_{4,k}
\end{matrix}\right]
\end{aligned}
\ees
is a rotation matrix, and $v_k \in \mathcal{V} = \left\{0.15{\rm I}_6,0_{6 \times 1}\right\}$.

To simulate the system, we consider the uncorrupted signal $\check{u}_k$, initial state $x_0 = \bb 0 & 1 & 0 & 0 \eb^{\top} \in \mathcal{X}_0 = \left\{0.18{\rm I}_4,\bb 0.1 & 0.9 & 0.1 & 0.1 \eb^{\top}\right\}$, realizations of uniform noise defined in $\mathcal{V}$ for $v_k$, $k_f = 200$, and $m_s = 5$. \AP{To estimate states, we consider the corrupted signal $u_k$, fixed values $\varphi_c = 10$ and $\varphi_g = 30$, the invariant $g\left(x_k\right) = x^{\top}_kx_k - 1$, and the feasible set $\mathcal{X}^{\rm F} = \left\{{\rm I}_4,0_{4 \times 1}\right\}$.} Since $\check{u}_k$ is truly unknown, the algorithms replace $\check{u}_k$ by $(u_k - w_k)$. 
\AP{Due to the nonlinearity of $x_k = f\left(x_{k-1},u_{k-1},w_{k-1}\right)$ in \eqref{eq_process_attitude}, Proposition \ref{propo_quad_DC_fun} is employed to yield DC functions $f = f^{\rm a} - f^{\rm b}$ over each time step. In this case, the polytope $\mathcal{P}_{k-1}$ related to Theorem \ref{theo_forecast} is a box (given by Proposition \ref{propo_int_hull_cz}) that contains the CZ $\mathcal{Z}_{k-1} = \mathcal{X}_{k-1} \times u_{k-1} \times \mathcal{W}_{k-1}$.
By exploiting the quadratic nature of both $y_k = h\left(x_k\right) + v_k$ in \eqref{eq_measurement_attitude} and $g\left(x_k\right) = 0$, we propose the DC functions $y_k = h^{\rm a} - h^{\rm b} + v_k$ and $g = g^{\rm a} - g^{\rm b}$ such that
\begin{align*}
h^{\rm a} &= \bb x^2_{1,k} + x^2_{4,k} \\ x_{1,k}x_{2,k} - x_{3,k}x_{4,k} \\ x_{1,k}x_{3,k} + x_{2,k}x_{4,k} \\ x_{1,k}x_{2,k} + x_{3,k}x_{4,k} \\ x^2_{2,k} + x^2_{4,k} \\ x_{2,k}x_{3,k} - x_{1,k}x_{4,k} \eb,\quad h^{\rm b} = \bb x^2_{2,k} + x^2_{3,k} \\ -x_{1,k}x_{2,k} + x_{3,k}x_{4,k} \\ -x_{1,k}x_{3,k} - x_{2,k}x_{4,k} \\ -x_{1,k}x_{2,k} - x_{3,k}x_{4,k} \\ x^2_{1,k} + x^2_{3,k} \\ x_{1,k}x_{4,k} - x_{2,k}x_{3,k} \eb, \\
g^{\rm a} &= g,\quad g^{\rm b} = 0.
\end{align*}
In this case, the polytopes $\mathcal{P}_k$ and $\check{\mathcal{P}}_k$, in Theorem \ref{theo_data_ass} and Corollary \ref{coro_consistency}, are parallelotopes (given by Proposition \ref{propo_partope}) that contain the CZs $\mathcal{Z}_k = \mathcal{X}_{k|k-1}$ and $\check{\mathcal{Z}}_k = \check{\mathcal{X}}_k$, respectively.
}

\AP{Figure \ref{fig_estimation_attitude_ref} depicts a separate simulation with the CZDC, CZFO and CZMV algorithms. Boxes were sketched rather than CZs for computational simplicity. According to the figure, CZDC generates CZs with the smallest associated interval hulls. Moreover, a faster reduction of uncertainty is expected with CZDC during the initialization effect. Table \ref{tab_attitude_ref} corroborates the improvement of precision caused by CZDC in comparison with both CZMV and CZFO. Since CZFO is, in general, more costly than CZMV \cite[Table 1]{rego2021set}, it demands a larger $T^{\rm CPU}$ as shown in Table \ref{tab_attitude_ref}. Differently, CZDC can enhance the precision of CZMV using much less computational resource, and this advantage is related to both tight linearization remainder (Lemma \ref{lemm_lin_remainder_DC}) and low-dimension CZs (Remark \ref{rema_order_comparison}). However, the quantity of operations involved with CZDC may be larger than the CZMV one, justifying the difference of $T^{\rm CPU}$.}


\AP{In order to verify if the precision of CZDC would be enlarged with respect to Table \ref{tab_attitude_ref} (reduction of $A^{\Box}$), we also tested if convexifying each row of $f$ or $-f$, for each time step, would be better (Remark \ref{rema_improved_quad_function}), selecting the strategy with the smallest lower bound of eigenvalue. However, the tests pointed out that convexifying $f$ always yielded the best solutions.}

\AP{During the experiment execution, CZMV and CZFO diverged for some simulations, whose results were discarded and not included in the computation of $A^{\Box}$. The increase of $\varphi_c$ and $\varphi_g$ can, in principle, improve the results. However, the generator reduction can imply conservatism for some directions due to the wrapping effect. Each simulation has different noise realizations, which affect $u_k$ and $y_k$, and thereby, the intersections.}

\begin{figure*}[!tb]
	\center
	\def\svgwidth{15cm}
	{\scriptsize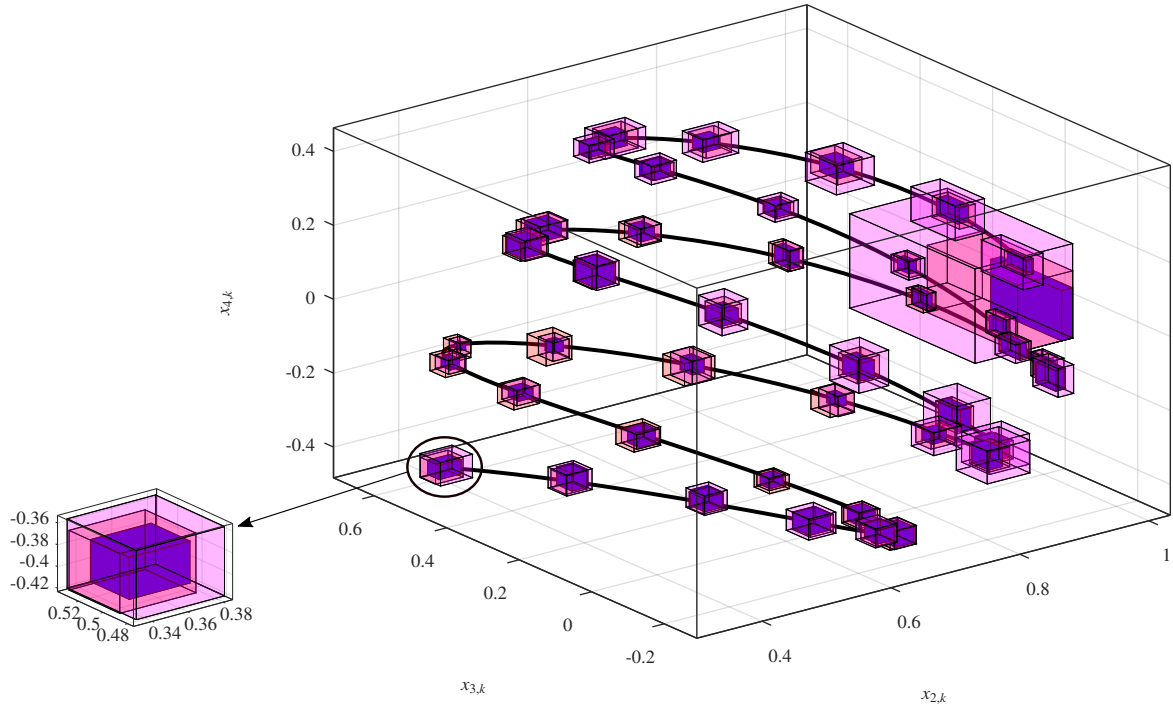}
	\caption{\AP{3D projection with respect to $x_{1,k} = 0$ of the interval hulls of CZs, computed by CZDC (blue color), CZFO (red color), and CZMV (magenta color), for the second case study (Subsection \ref{subsec_second_exam}). The true states are sketched in black solid line.}}
	\label{fig_estimation_attitude_ref}
\end{figure*}

\begin{table}[t]
	\centering
	\caption{Results of $T^{\rm CPU}$ and \AP{$A^{\Box}$} for the second example (Subsection \ref{subsec_second_exam}). \AP{The percentage reduction of $A^{\Box}$ for CZDC and CZFO in comparison to CZMV is shown between parenthesis.}}
	\label{tab_attitude_ref}
	\begin{tabular}{c|c|c|c}
	\hline
	Indexes         & CZDC & CZFO & CZMV  \\ \hline
    $T^{\rm CPU}$   & 3.58 s & 39.8 s & {\bf 1.76} s \\
	\AP{$A^{\Box} (\times 10^{-4})$} & {\bf 0.0265} ($\downarrow$96.7\%) & 0.1714 ($\downarrow$78.4\%) & 0.7919 \\ \hline
	\end{tabular}
\end{table}

%% file: figures/fig_area_alamo2008_czdc.pdf_tex
\begingroup%
  \makeatletter%
  \providecommand\color[2][]{%
    \errmessage{(Inkscape) Color is used for the text in Inkscape, but the package 'color.sty' is not loaded}%
    \renewcommand\color[2][]{}%
  }%
  \providecommand\transparent[1]{%
    \errmessage{(Inkscape) Transparency is used (non-zero) for the text in Inkscape, but the package 'transparent.sty' is not loaded}%
    \renewcommand\transparent[1]{}%
  }%
  \providecommand\rotatebox[2]{#2}%
  \newcommand*\fsize{\dimexpr\f@size pt\relax}%
  \newcommand*\lineheight[1]{\fontsize{\fsize}{#1\fsize}\selectfont}%
  \ifx\svgwidth\undefined%
    \setlength{\unitlength}{345.34562112bp}%
    \ifx\svgscale\undefined%
      \relax%
    \else%
      \setlength{\unitlength}{\unitlength * \real{\svgscale}}%
    \fi%
  \else%
    \setlength{\unitlength}{\svgwidth}%
  \fi%
  \global\let\svgwidth\undefined%
  \global\let\svgscale\undefined%
  \makeatother%
  \begin{picture}(1,0.46197919)%
    \lineheight{1}%
    \setlength\tabcolsep{0pt}%
    \put(0,0){\includegraphics[width=\unitlength,page=1]{fig_area_alamo2008_czdc.pdf}}%
    \put(0.03162168,0.25794785){\color[rgb]{0.14901961,0.14901961,0.14901961}\rotatebox{90}{\makebox(0,0)[t]{\lineheight{1.25}\smash{\begin{tabular}[t]{c}log(Area)\end{tabular}}}}}%
    \put(0.5375822,0.00259405){\color[rgb]{0.14901961,0.14901961,0.14901961}\makebox(0,0)[t]{\lineheight{1.25}\smash{\begin{tabular}[t]{c}$k$\end{tabular}}}}%
    \put(0,0){\includegraphics[width=\unitlength,page=2]{fig_area_alamo2008_czdc.pdf}}%
    \put(0.09576829,0.03427455){\color[rgb]{0.14901961,0.14901961,0.14901961}\makebox(0,0)[t]{\lineheight{1.25}\smash{\begin{tabular}[t]{c}0\end{tabular}}}}%
    \put(0.31667525,0.03427455){\color[rgb]{0.14901961,0.14901961,0.14901961}\makebox(0,0)[t]{\lineheight{1.25}\smash{\begin{tabular}[t]{c}10\end{tabular}}}}%
    \put(0.5375822,0.03427455){\color[rgb]{0.14901961,0.14901961,0.14901961}\makebox(0,0)[t]{\lineheight{1.25}\smash{\begin{tabular}[t]{c}20\end{tabular}}}}%
    \put(0.75848699,0.03427455){\color[rgb]{0.14901961,0.14901961,0.14901961}\makebox(0,0)[t]{\lineheight{1.25}\smash{\begin{tabular}[t]{c}30\end{tabular}}}}%
    \put(0.97939394,0.03427455){\color[rgb]{0.14901961,0.14901961,0.14901961}\makebox(0,0)[t]{\lineheight{1.25}\smash{\begin{tabular}[t]{c}40\end{tabular}}}}%
    \put(0,0){\includegraphics[width=\unitlength,page=3]{fig_area_alamo2008_czdc.pdf}}%
    \put(0.07809469,0.05494336){\color[rgb]{0.14901961,0.14901961,0.14901961}\makebox(0,0)[rt]{\lineheight{1.25}\smash{\begin{tabular}[t]{r}-1\end{tabular}}}}%
    \put(0.07809469,0.11067914){\color[rgb]{0.14901961,0.14901961,0.14901961}\makebox(0,0)[rt]{\lineheight{1.25}\smash{\begin{tabular}[t]{r}0\end{tabular}}}}%
    \put(0.07809469,0.16641366){\color[rgb]{0.14901961,0.14901961,0.14901961}\makebox(0,0)[rt]{\lineheight{1.25}\smash{\begin{tabular}[t]{r}1\end{tabular}}}}%
    \put(0.07809469,0.22214818){\color[rgb]{0.14901961,0.14901961,0.14901961}\makebox(0,0)[rt]{\lineheight{1.25}\smash{\begin{tabular}[t]{r}2\end{tabular}}}}%
    \put(0.07809469,0.27788397){\color[rgb]{0.14901961,0.14901961,0.14901961}\makebox(0,0)[rt]{\lineheight{1.25}\smash{\begin{tabular}[t]{r}3\end{tabular}}}}%
    \put(0.07809469,0.33361849){\color[rgb]{0.14901961,0.14901961,0.14901961}\makebox(0,0)[rt]{\lineheight{1.25}\smash{\begin{tabular}[t]{r}4\end{tabular}}}}%
    \put(0.07809469,0.38935301){\color[rgb]{0.14901961,0.14901961,0.14901961}\makebox(0,0)[rt]{\lineheight{1.25}\smash{\begin{tabular}[t]{r}5\end{tabular}}}}%
    \put(0.07809469,0.44508879){\color[rgb]{0.14901961,0.14901961,0.14901961}\makebox(0,0)[rt]{\lineheight{1.25}\smash{\begin{tabular}[t]{r}6\end{tabular}}}}%
    \put(0,0){\includegraphics[width=\unitlength,page=4]{fig_area_alamo2008_czdc.pdf}}%
    \put(0.64879822,0.3197146){\color[rgb]{0,0,0}\makebox(0,0)[lt]{\lineheight{1.25}\smash{\begin{tabular}[t]{l}\textbf{CZDC}\end{tabular}}}}%
    \put(0.64908102,0.21884645){\color[rgb]{0,0,0}\makebox(0,0)[lt]{\lineheight{1.25}\smash{\begin{tabular}[t]{l}CZMV\end{tabular}}}}%
    \put(0.64734737,0.27096391){\color[rgb]{0,0,0}\makebox(0,0)[lt]{\lineheight{1.25}\smash{\begin{tabular}[t]{l}CZFO\end{tabular}}}}%
    \put(0,0){\includegraphics[width=\unitlength,page=5]{fig_area_alamo2008_czdc.pdf}}%
    \put(0.65113316,0.36819891){\color[rgb]{0,0,0}\makebox(0,0)[lt]{\lineheight{1.25}\smash{\begin{tabular}[t]{l}ZDC\end{tabular}}}}%
  \end{picture}%
\endgroup%

%% file: figures/fig_x1_alamo2008_czdc.pdf_tex
\begingroup%
  \makeatletter%
  \providecommand\color[2][]{%
    \errmessage{(Inkscape) Color is used for the text in Inkscape, but the package 'color.sty' is not loaded}%
    \renewcommand\color[2][]{}%
  }%
  \providecommand\transparent[1]{%
    \errmessage{(Inkscape) Transparency is used (non-zero) for the text in Inkscape, but the package 'transparent.sty' is not loaded}%
    \renewcommand\transparent[1]{}%
  }%
  \providecommand\rotatebox[2]{#2}%
  \newcommand*\fsize{\dimexpr\f@size pt\relax}%
  \newcommand*\lineheight[1]{\fontsize{\fsize}{#1\fsize}\selectfont}%
  \ifx\svgwidth\undefined%
    \setlength{\unitlength}{346.15128518bp}%
    \ifx\svgscale\undefined%
      \relax%
    \else%
      \setlength{\unitlength}{\unitlength * \real{\svgscale}}%
    \fi%
  \else%
    \setlength{\unitlength}{\svgwidth}%
  \fi%
  \global\let\svgwidth\undefined%
  \global\let\svgscale\undefined%
  \makeatother%
  \begin{picture}(1,0.47919464)%
    \lineheight{1}%
    \setlength\tabcolsep{0pt}%
    \put(0,0){\includegraphics[width=\unitlength,page=1]{fig_x1_alamo2008_czdc.pdf}}%
    \put(0.03387558,0.26760004){\color[rgb]{0.14901961,0.14901961,0.14901961}\rotatebox{90}{\makebox(0,0)[t]{\lineheight{1.25}\smash{\begin{tabular}[t]{c}$x_{1,k}$\end{tabular}}}}}%
    \put(0.53865848,0.00269111){\color[rgb]{0.14901961,0.14901961,0.14901961}\makebox(0,0)[t]{\lineheight{1.25}\smash{\begin{tabular}[t]{c}$k$\end{tabular}}}}%
    \put(0,0){\includegraphics[width=\unitlength,page=2]{fig_x1_alamo2008_czdc.pdf}}%
    \put(0.09787288,0.03555707){\color[rgb]{0.14901961,0.14901961,0.14901961}\makebox(0,0)[t]{\lineheight{1.25}\smash{\begin{tabular}[t]{c}0\end{tabular}}}}%
    \put(0.31826568,0.03555707){\color[rgb]{0.14901961,0.14901961,0.14901961}\makebox(0,0)[t]{\lineheight{1.25}\smash{\begin{tabular}[t]{c}10\end{tabular}}}}%
    \put(0.53865848,0.03555707){\color[rgb]{0.14901961,0.14901961,0.14901961}\makebox(0,0)[t]{\lineheight{1.25}\smash{\begin{tabular}[t]{c}20\end{tabular}}}}%
    \put(0.75904911,0.03555707){\color[rgb]{0.14901961,0.14901961,0.14901961}\makebox(0,0)[t]{\lineheight{1.25}\smash{\begin{tabular}[t]{c}30\end{tabular}}}}%
    \put(0.97944191,0.03555707){\color[rgb]{0.14901961,0.14901961,0.14901961}\makebox(0,0)[t]{\lineheight{1.25}\smash{\begin{tabular}[t]{c}40\end{tabular}}}}%
    \put(0,0){\includegraphics[width=\unitlength,page=3]{fig_x1_alamo2008_czdc.pdf}}%
    \put(0.08024042,0.05699929){\color[rgb]{0.14901961,0.14901961,0.14901961}\makebox(0,0)[rt]{\lineheight{1.25}\smash{\begin{tabular}[t]{r}-3\end{tabular}}}}%
    \put(0.08024042,0.11482066){\color[rgb]{0.14901961,0.14901961,0.14901961}\makebox(0,0)[rt]{\lineheight{1.25}\smash{\begin{tabular}[t]{r}-2\end{tabular}}}}%
    \put(0.08024042,0.17264072){\color[rgb]{0.14901961,0.14901961,0.14901961}\makebox(0,0)[rt]{\lineheight{1.25}\smash{\begin{tabular}[t]{r}-1\end{tabular}}}}%
    \put(0.08024042,0.23046078){\color[rgb]{0.14901961,0.14901961,0.14901961}\makebox(0,0)[rt]{\lineheight{1.25}\smash{\begin{tabular}[t]{r}0\end{tabular}}}}%
    \put(0.08024042,0.28828215){\color[rgb]{0.14901961,0.14901961,0.14901961}\makebox(0,0)[rt]{\lineheight{1.25}\smash{\begin{tabular}[t]{r}1\end{tabular}}}}%
    \put(0.08024042,0.34610222){\color[rgb]{0.14901961,0.14901961,0.14901961}\makebox(0,0)[rt]{\lineheight{1.25}\smash{\begin{tabular}[t]{r}2\end{tabular}}}}%
    \put(0.08024042,0.40392228){\color[rgb]{0.14901961,0.14901961,0.14901961}\makebox(0,0)[rt]{\lineheight{1.25}\smash{\begin{tabular}[t]{r}3\end{tabular}}}}%
    \put(0.08024042,0.46174365){\color[rgb]{0.14901961,0.14901961,0.14901961}\makebox(0,0)[rt]{\lineheight{1.25}\smash{\begin{tabular}[t]{r}4\end{tabular}}}}%
    \put(0,0){\includegraphics[width=\unitlength,page=4]{fig_x1_alamo2008_czdc.pdf}}%
    \put(0.79621803,0.41510098){\color[rgb]{0,0,0}\makebox(0,0)[lt]{\lineheight{1.25}\smash{\begin{tabular}[t]{l}True state\end{tabular}}}}%
  \end{picture}%
\endgroup%

%% file: figures/fig_x2_alamo2008_czdc.pdf_tex
\begingroup%
  \makeatletter%
  \providecommand\color[2][]{%
    \errmessage{(Inkscape) Color is used for the text in Inkscape, but the package 'color.sty' is not loaded}%
    \renewcommand\color[2][]{}%
  }%
  \providecommand\transparent[1]{%
    \errmessage{(Inkscape) Transparency is used (non-zero) for the text in Inkscape, but the package 'transparent.sty' is not loaded}%
    \renewcommand\transparent[1]{}%
  }%
  \providecommand\rotatebox[2]{#2}%
  \newcommand*\fsize{\dimexpr\f@size pt\relax}%
  \newcommand*\lineheight[1]{\fontsize{\fsize}{#1\fsize}\selectfont}%
  \ifx\svgwidth\undefined%
    \setlength{\unitlength}{346.15128518bp}%
    \ifx\svgscale\undefined%
      \relax%
    \else%
      \setlength{\unitlength}{\unitlength * \real{\svgscale}}%
    \fi%
  \else%
    \setlength{\unitlength}{\svgwidth}%
  \fi%
  \global\let\svgwidth\undefined%
  \global\let\svgscale\undefined%
  \makeatother%
  \begin{picture}(1,0.47081658)%
    \lineheight{1}%
    \setlength\tabcolsep{0pt}%
    \put(0,0){\includegraphics[width=\unitlength,page=1]{fig_x2_alamo2008_czdc.pdf}}%
    \put(0.03387558,0.26288224){\color[rgb]{0.14901961,0.14901961,0.14901961}\rotatebox{90}{\makebox(0,0)[t]{\lineheight{1.25}\smash{\begin{tabular}[t]{c}$x_{2,k}$\end{tabular}}}}}%
    \put(0.53865848,0.00264367){\color[rgb]{0.14901961,0.14901961,0.14901961}\makebox(0,0)[t]{\lineheight{1.25}\smash{\begin{tabular}[t]{c}$k$\end{tabular}}}}%
    \put(0,0){\includegraphics[width=\unitlength,page=2]{fig_x2_alamo2008_czdc.pdf}}%
    \put(0.09787288,0.0349302){\color[rgb]{0.14901961,0.14901961,0.14901961}\makebox(0,0)[t]{\lineheight{1.25}\smash{\begin{tabular}[t]{c}0\end{tabular}}}}%
    \put(0.31826568,0.0349302){\color[rgb]{0.14901961,0.14901961,0.14901961}\makebox(0,0)[t]{\lineheight{1.25}\smash{\begin{tabular}[t]{c}10\end{tabular}}}}%
    \put(0.53865848,0.0349302){\color[rgb]{0.14901961,0.14901961,0.14901961}\makebox(0,0)[t]{\lineheight{1.25}\smash{\begin{tabular}[t]{c}20\end{tabular}}}}%
    \put(0.75904911,0.0349302){\color[rgb]{0.14901961,0.14901961,0.14901961}\makebox(0,0)[t]{\lineheight{1.25}\smash{\begin{tabular}[t]{c}30\end{tabular}}}}%
    \put(0.97944191,0.0349302){\color[rgb]{0.14901961,0.14901961,0.14901961}\makebox(0,0)[t]{\lineheight{1.25}\smash{\begin{tabular}[t]{c}40\end{tabular}}}}%
    \put(0,0){\includegraphics[width=\unitlength,page=3]{fig_x2_alamo2008_czdc.pdf}}%
    \put(0.08024042,0.05599439){\color[rgb]{0.14901961,0.14901961,0.14901961}\makebox(0,0)[rt]{\lineheight{1.25}\smash{\begin{tabular}[t]{r}-2\end{tabular}}}}%
    \put(0.08024042,0.15539689){\color[rgb]{0.14901961,0.14901961,0.14901961}\makebox(0,0)[rt]{\lineheight{1.25}\smash{\begin{tabular}[t]{r}0\end{tabular}}}}%
    \put(0.08024042,0.25479938){\color[rgb]{0.14901961,0.14901961,0.14901961}\makebox(0,0)[rt]{\lineheight{1.25}\smash{\begin{tabular}[t]{r}2\end{tabular}}}}%
    \put(0.08024042,0.35420059){\color[rgb]{0.14901961,0.14901961,0.14901961}\makebox(0,0)[rt]{\lineheight{1.25}\smash{\begin{tabular}[t]{r}4\end{tabular}}}}%
    \put(0.08024042,0.45360309){\color[rgb]{0.14901961,0.14901961,0.14901961}\makebox(0,0)[rt]{\lineheight{1.25}\smash{\begin{tabular}[t]{r}6\end{tabular}}}}%
    \put(0,0){\includegraphics[width=\unitlength,page=4]{fig_x2_alamo2008_czdc.pdf}}%
    \put(0.76223207,0.40040498){\color[rgb]{0,0,0}\makebox(0,0)[lt]{\lineheight{1.25}\smash{\begin{tabular}[t]{l}True state\end{tabular}}}}%
  \end{picture}%
\endgroup%

%% file: figures/fig_x234_attitude_czdc.pdf_tex
\begingroup%
  \makeatletter%
  \providecommand\color[2][]{%
    \errmessage{(Inkscape) Color is used for the text in Inkscape, but the package 'color.sty' is not loaded}%
    \renewcommand\color[2][]{}%
  }%
  \providecommand\transparent[1]{%
    \errmessage{(Inkscape) Transparency is used (non-zero) for the text in Inkscape, but the package 'transparent.sty' is not loaded}%
    \renewcommand\transparent[1]{}%
  }%
  \providecommand\rotatebox[2]{#2}%
  \newcommand*\fsize{\dimexpr\f@size pt\relax}%
  \newcommand*\lineheight[1]{\fontsize{\fsize}{#1\fsize}\selectfont}%
  \ifx\svgwidth\undefined%
    \setlength{\unitlength}{737.53618636bp}%
    \ifx\svgscale\undefined%
      \relax%
    \else%
      \setlength{\unitlength}{\unitlength * \real{\svgscale}}%
    \fi%
  \else%
    \setlength{\unitlength}{\svgwidth}%
  \fi%
  \global\let\svgwidth\undefined%
  \global\let\svgscale\undefined%
  \makeatother%
  \begin{picture}(1,0.62194636)%
    \lineheight{1}%
    \setlength\tabcolsep{0pt}%
    \put(0,0){\includegraphics[width=\unitlength,page=1]{fig_x234_attitude_czdc.pdf}}%
    \put(0.17368925,0.35789653){\color[rgb]{0.14901961,0.14901961,0.14901961}\rotatebox{90}{\makebox(0,0)[t]{\lineheight{1.25}\smash{\begin{tabular}[t]{c}$x_{4,k}$\end{tabular}}}}}%
    \put(0.39992467,0.01558914){\color[rgb]{0.14901961,0.14901961,0.14901961}\makebox(0,0)[rt]{\lineheight{1.25}\smash{\begin{tabular}[t]{r}$x_{3,k}$\end{tabular}}}}%
    \put(0.78051088,0.00734437){\color[rgb]{0.14901961,0.14901961,0.14901961}\makebox(0,0)[lt]{\lineheight{1.25}\smash{\begin{tabular}[t]{l}$x_{2,k}$\end{tabular}}}}%
    \put(0,0){\includegraphics[width=\unitlength,page=2]{fig_x234_attitude_czdc.pdf}}%
    \put(0.65725829,0.04051018){\color[rgb]{0.14901961,0.14901961,0.14901961}\makebox(0,0)[t]{\lineheight{1.25}\smash{\begin{tabular}[t]{c}0.4\end{tabular}}}}%
    \put(0.77032153,0.07011808){\color[rgb]{0.14901961,0.14901961,0.14901961}\makebox(0,0)[t]{\lineheight{1.25}\smash{\begin{tabular}[t]{c}0.6\end{tabular}}}}%
    \put(0.88338659,0.09972599){\color[rgb]{0.14901961,0.14901961,0.14901961}\makebox(0,0)[t]{\lineheight{1.25}\smash{\begin{tabular}[t]{c}0.8\end{tabular}}}}%
    \put(0.99645164,0.1293339){\color[rgb]{0.14901961,0.14901961,0.14901961}\makebox(0,0)[t]{\lineheight{1.25}\smash{\begin{tabular}[t]{c}1\end{tabular}}}}%
    \put(0,0){\includegraphics[width=\unitlength,page=3]{fig_x234_attitude_czdc.pdf}}%
    \put(0.5346557,0.03989771){\color[rgb]{0.14901961,0.14901961,0.14901961}\makebox(0,0)[t]{\lineheight{1.25}\smash{\begin{tabular}[t]{c}-0.2\end{tabular}}}}%
    \put(0.47104492,0.06818994){\color[rgb]{0.14901961,0.14901961,0.14901961}\makebox(0,0)[t]{\lineheight{1.25}\smash{\begin{tabular}[t]{c}0\end{tabular}}}}%
    \put(0.40743232,0.09648217){\color[rgb]{0.14901961,0.14901961,0.14901961}\makebox(0,0)[t]{\lineheight{1.25}\smash{\begin{tabular}[t]{c}0.2\end{tabular}}}}%
    \put(0.34382154,0.12477266){\color[rgb]{0.14901961,0.14901961,0.14901961}\makebox(0,0)[t]{\lineheight{1.25}\smash{\begin{tabular}[t]{c}0.4\end{tabular}}}}%
    \put(0.28020894,0.15306489){\color[rgb]{0.14901961,0.14901961,0.14901961}\makebox(0,0)[t]{\lineheight{1.25}\smash{\begin{tabular}[t]{c}0.6\end{tabular}}}}%
    \put(0,0){\includegraphics[width=\unitlength,page=4]{fig_x234_attitude_czdc.pdf}}%
    \put(0.24942303,0.22652454){\color[rgb]{0.14901961,0.14901961,0.14901961}\makebox(0,0)[rt]{\lineheight{1.25}\smash{\begin{tabular}[t]{r}-0.4\end{tabular}}}}%
    \put(0.24942303,0.29185631){\color[rgb]{0.14901961,0.14901961,0.14901961}\makebox(0,0)[rt]{\lineheight{1.25}\smash{\begin{tabular}[t]{r}-0.2\end{tabular}}}}%
    \put(0.24942303,0.35718984){\color[rgb]{0.14901961,0.14901961,0.14901961}\makebox(0,0)[rt]{\lineheight{1.25}\smash{\begin{tabular}[t]{r}0\end{tabular}}}}%
    \put(0.24942303,0.42252162){\color[rgb]{0.14901961,0.14901961,0.14901961}\makebox(0,0)[rt]{\lineheight{1.25}\smash{\begin{tabular}[t]{r}0.2\end{tabular}}}}%
    \put(0.24942303,0.48785339){\color[rgb]{0.14901961,0.14901961,0.14901961}\makebox(0,0)[rt]{\lineheight{1.25}\smash{\begin{tabular}[t]{r}0.4\end{tabular}}}}%
    \put(0,0){\includegraphics[width=\unitlength,page=5]{fig_x234_attitude_czdc.pdf}}%
    \put(0.11772382,0.06034842){\color[rgb]{0.14901961,0.14901961,0.14901961}\makebox(0,0)[t]{\lineheight{1.25}\smash{\begin{tabular}[t]{c}0.34\end{tabular}}}}%
    \put(0.14900462,0.06885962){\color[rgb]{0.14901961,0.14901961,0.14901961}\makebox(0,0)[t]{\lineheight{1.25}\smash{\begin{tabular}[t]{c}0.36\end{tabular}}}}%
    \put(0.18028542,0.07737082){\color[rgb]{0.14901961,0.14901961,0.14901961}\makebox(0,0)[t]{\lineheight{1.25}\smash{\begin{tabular}[t]{c}0.38\end{tabular}}}}%
    \put(0,0){\includegraphics[width=\unitlength,page=6]{fig_x234_attitude_czdc.pdf}}%
    \put(0.07218332,0.05871014){\color[rgb]{0.14901961,0.14901961,0.14901961}\makebox(0,0)[t]{\lineheight{1.25}\smash{\begin{tabular}[t]{c}0.48\end{tabular}}}}%
    \put(0.05073888,0.06862021){\color[rgb]{0.14901961,0.14901961,0.14901961}\makebox(0,0)[t]{\lineheight{1.25}\smash{\begin{tabular}[t]{c}0.5\end{tabular}}}}%
    \put(0.02929482,0.0785299){\color[rgb]{0.14901961,0.14901961,0.14901961}\makebox(0,0)[t]{\lineheight{1.25}\smash{\begin{tabular}[t]{c}0.52\end{tabular}}}}%
    \put(0,0){\includegraphics[width=\unitlength,page=7]{fig_x234_attitude_czdc.pdf}}%
    \put(0.01585252,0.10473087){\color[rgb]{0.14901961,0.14901961,0.14901961}\makebox(0,0)[rt]{\lineheight{1.25}\smash{\begin{tabular}[t]{r}-0.42\end{tabular}}}}%
    \put(0.01585252,0.12396027){\color[rgb]{0.14901961,0.14901961,0.14901961}\makebox(0,0)[rt]{\lineheight{1.25}\smash{\begin{tabular}[t]{r}-0.4\end{tabular}}}}%
    \put(0.01585252,0.1431893){\color[rgb]{0.14901961,0.14901961,0.14901961}\makebox(0,0)[rt]{\lineheight{1.25}\smash{\begin{tabular}[t]{r}-0.38\end{tabular}}}}%
    \put(0.01585252,0.16241834){\color[rgb]{0.14901961,0.14901961,0.14901961}\makebox(0,0)[rt]{\lineheight{1.25}\smash{\begin{tabular}[t]{r}-0.36\end{tabular}}}}%
    \put(0,0){\includegraphics[width=\unitlength,page=8]{fig_x234_attitude_czdc.pdf}}%
  \end{picture}%
\endgroup%

%% file: 06Conclusions_version01b.tex
\section{Conclusions} \label{sec_conclusion}

This paper proposed a new set-membership filter for discrete-time nonlinear uncertain systems with state constraints, called CZDC.
A DC programming approach \AP{was} used to provide a new nonlinear approximation for CZs. Thus, CZDC \AP{established} an alternative estimation basis with respect to the state-of-the-art algorithms, called CZMV and CZFO \cite{rego2021set}. We showed that the performance of these two algorithms can be significantly deteriorated due to the wrapping and dependency effects, with CZDC being a good option to mitigate divergence and conservatism issues. Over two numerical examples, we discussed advantages of CZDC over CZMV and CZFO.
These three algorithms can readily enforce linear inequality constraints on the state vector by using CZs. However, the nonlinear case requires investigation and will be intended in the future.